\documentclass[11pt]{amsart}

\usepackage{amsthm, amsmath, amssymb}
\usepackage[T1]{fontenc}   
\usepackage[utf8]{inputenc} 

\usepackage{comment}
\usepackage{xcolor}
\usepackage{bm}  

\usepackage{hyperref}

\usepackage[letterpaper, top=3cm, bottom=3cm, left=3cm, right=3cm]{geometry}

\newtheorem{thm}{Theorem}[section]
\newtheorem{cor}[thm]{Corollary}
\newtheorem{prop}[thm]{Proposition}
\newtheorem{definition}[thm]{Definition}
\newtheorem{lem}[thm]{Lemma}
\newtheorem{rem}[thm]{Remark}

\newcommand{\N}{\mathbb{N}}

\newcommand{\R}{\mathbb{R}}
\newcommand{\T}{\mathbb{T}}
\newcommand{\C}{\mathbb{C}}
\newcommand{\Z}{\mathbb{Z}}
\newcommand{\ind}[1]{\mathbf{1}_{#1}}  

\DeclareMathOperator{\supp}{supp}

\newcommand{\E}{\mathbb{E}}
\newcommand{\Prob}{\mathbb{P}}

\numberwithin{equation}{section}

\allowdisplaybreaks

\begin{document}

\begin{abstract}
In this paper, we present an exposition of the work \cite{B} by Jean Bourgain, in which he resolved the well known conjecture posed by Rudin regarding the existence of $\Lambda(p)$-sets. 
\end{abstract}


\title[Exposition on the $\Lambda(p)$-set Problem] {Expository article: ``Bounded orthogonal systems and the $\Lambda(p)$-set problem'' by Jean Bourgain}

\author[H. Jung]{Hongki Jung}
\address{Hongki Jung, \newline
Department of Mathematics, \newline
Louisiana State University, \newline
147 Lockett Hall, Baton Rouge, LA 70803, USA
}
\email{hjung@lsu.edu}

\author[B. Langowski]{Bartosz Langowski}

\address{Bartosz Langowski, \newline
Department of Mathematics and Physical Sciences, \newline
Franciscan University of Steubenville, \newline
1235 University Blvd.,\ Steubenville, OH 43952, USA
}
\email{blangowski@franciscan.edu}

\author[A. Ortiz]
{Alexander Ortiz}
\address{Alexander Ortiz, \newline
Department of Mathematics, \newline
Rice University, \newline
6100 Main St, Houston, TX 77005, USA
}
\email{ao80@rice.edu}

\author[T. Vu]
{Truong Vu}
\address{Truong Vu, \newline
Department of Mathematics, Statistics, and Computer Science, \newline
University of Illinois at Chicago, \newline
851 S. Morgan Street, Chicago, IL 60607, USA
}
\email{tvu25@uic.edu}

\maketitle

\setcounter{tocdepth}{1}

\tableofcontents

\section{Introduction and statement of the main result}
The purpose of this article is to present an exposition of the paper \cite{B} of Jean Bourgain, where he proves the conjecture posed by Rudin about the existence of $\Lambda(p)$-sets which are not $\Lambda(q)$-sets for any $q<p$.
The methods used in his proof are deeply insightful and have far-reaching implications in analysis and related fields. His argument involves sophisticated probabilistic techniques, metric entropy methods and chaining arguments.

We begin with a short introduction to $\Lambda(p)$-sets.
\begin{definition}\label{def:Lambda_p_set}
 For $0< p<\infty$ and an integer $d\ge 1$ we say that $S\subset \Z^d$ is a \emph{$\Lambda(p)$-set} (equivalently: has the \emph{$\Lambda(p)$-property}) if there exists $q<p$ such that
\begin{align}\label{eq:lambdapq}
\|\sum_{j\in S} a_j e^{2\pi i j\cdot \theta}\|_{L^p(\T^d)}\lesssim\|\sum_{j\in S} a_j e^{2\pi i j\cdot \theta}\|_{L^q(\T^d)}
\end{align}
for all finitely supported multiparameter sequences $(a_j)_{j\in\Z^d}$. 
\end{definition}
It is not difficult to see, with the aid of H\"older's inequality, that the $\Lambda(p)$-property from Definition \ref{def:Lambda_p_set} depends only on $p$ and not on $q$. In particular, if $p>2$, then one can show that \eqref{eq:lambdapq} is equivalent to
\begin{align}\label{eq:lambdap2}
\|\sum_{j\in S} a_j e^{2\pi i j \cdot\theta}\|_{L^p(\T^d)}\lesssim\|\sum_{j\in\Z^d} a_j e^{2\pi i j \cdot \theta}\|_{L^2(\T^d)}.
\end{align}
Note that due to H\"older's inequality the $L^p(\T^d)$-norm dominates the $L^2(\T^d)$-norm if $p>2$, so \eqref{eq:lambdap2} can be thought of as a strengthening of Bessel's inequality
\begin{align*}
\|\sum_{j\in S} a_j e^{2\pi i j \cdot\theta}\|_{L^2(\T)}\le\|\sum_{j\in\Z} a_j e^{2\pi i j \cdot\theta}\|_{L^2(\T)}.
\end{align*}

The simplest example of a $\Lambda(p)$-set is the set of dyadic numbers $\{2^n\}_{n=0}^\infty$ which has the  $\Lambda(p)$-property for all $0<p<\infty$. More generally, any lacunary sequence $(b_n)_{n\in\Z}\subset \Z$, i.e., a sequence satisfying for some $\varepsilon>0$ the growth condition
$$
\frac{b_{n+1}}{b_n}>1+\varepsilon, \qquad n\in\Z,
$$
forms a $\Lambda(p)$-set for any $0<p<\infty$. In \cite{R} Rudin showed  the  $\Lambda(p)$-property for $p>1$ for Sidon sets, that is subsets $E$ of integers for which an estimate
\begin{align*}
\sum_{j\in E} |\widehat{f}(j)|\lesssim \|f\|_{L^\infty(\T)}
\end{align*}
holds for any bounded function $f$ on the torus.
Every lacunary set is a Sidon set. 

Another array of examples of $\Lambda(p)$-sets comes from discrete restriction theory. There has been substantial progress in this area due to the recent development of $\ell^2$ decoupling by Bourgain and Demeter---see \cite{BD}. In particular, decoupling establishes that the set of integer-lattice points on an elliptic paraboloid $\{(\mathbf{x},|\mathbf{x}|^2):\mathbf{x}\in ([1,N]\cap\Z)^{d-1}\}$ is a $\Lambda(p)$-set for $2\le p \le \frac{2(d+1)}{d-1}$ up to an $N^\varepsilon$--loss for any $\varepsilon>0$.
Decoupling for the moment curve, for example as proved by Bourgain, Demeter, and Guth \cite{BDG}, allows one to show that the integer-lattice points on the moment curve $\{(n,n^2,\dotsc, n^d): n\in [1,N]\cap\Z\}$ form a $\Lambda(p)$-set for $2\le p \le d(d+1)$ up to $N^\varepsilon$--loss.
The ranges of $p$ in both these theorems are sharp. It is conjectured that a discrete sphere $\{\mathbf{x}\in\Z^d: |\mathbf{x}|=N\}$ is a $\Lambda(p)$-set for $2\le p \le \frac{2d}{d-2}$. The decoupling techniques only yield partial progress on this problem. The difficulty in handling the case of a discrete sphere lies in its arithmetic features, which seem to be beyond the reach of purely analytic methods. 

As for negative examples, one can test \eqref{eq:lambdapq} with the Dirichlet kernel to see that the set of all integers is not a $\Lambda(p)$-set for any $0<p<\infty$. More involved arguments are needed to show that the set of  squares of integers is not a $\Lambda(p)$-set for $p<2$ or $p\ge4$. A famous conjecture of Rudin asserts that squares have the $\Lambda(p)$-property for $p\in(2,4)$.

Sets with the $\Lambda(p)$-property have some interesting structural properties. It was shown by Rudin \cite{R} that for $p>2$ a $\Lambda(p)$-set of size $n$ cannot contain an arithmetic progression of length larger than $C n^{2/p}$; in particular a Sidon set (which, as we already mentioned, belongs to $\bigcap_{p>1}\Lambda(p)$) of size $n$ can have at most $C\log n$ elements in an arithmetic progression. For more details in that direction and a nice summary of the $\Lambda(p)$-problem and its connections with Rudin's conjecture we encourage the reader to visit the blog of Ioannis Parissis \cite{P}.

It is straightforward to see that if $0<p_1<p_2$ then every $\Lambda(p_2)$-set is also a $\Lambda(p_1)$-set, which we will write shortly as $\Lambda(p_2)\subset\Lambda(p_1)$. That raises a natural question whether this inclusion is necessarily proper. This problem was first posed in the seminal paper of Rudin \cite{R}.

In the case $p\in(1,2)$ the negative answer was provided by Bachelis and Ebenstein in \cite{BE}. More precisely, they showed that for any $S\subset\Z$ the set
$\{ p\in(1,2): S \in\Lambda(p)\}$ is an open interval.

The situation turned out to be much more difficult in the case $p>2$. It had been a long standing open problem until Bourgain showed in his celebrated paper \cite{B} that the inclusion is indeed proper.
\begin{thm}\label{thm:main}
For any  $p>2$ there exists $S\subset\Z$ which is a $\Lambda(p)$-set, but is not a $\Lambda(r)$-set for any $r>p$.
\end{thm}
Bourgain proved the above result in the more general setting of $1$-bounded orthogonal systems. 
As we will see momentarily, Theorem \ref{thm:main} can be derived from the following result.
\begin{thm}\label{thm:main2}
For $n\in\N$ let $\Phi=(\varphi_j)_{j=1}^n$ be a sequence of mutually orthogonal real-valued functions on a probability space satisfying a uniform bound $\|\varphi_j\|_{L^\infty}\le 1$ for each $j\in[n]:=\{1, \dots, n\}$. For any $p>2$ there exists a subset of $[n]$ such that $|S|\simeq n^{2/p}$ and satisfying for any sequence of real coefficients $(a_i)_{i=1}^n$ the estimate
\begin{align}\label{eq:lambdapprop}
\| \sum_{i\in S}a_i\varphi_i \|_{L^p}\lesssim \big(\sum_{i\in S}|a_i|^2\big)^{1/2},
\end{align}
with the implicit constant depending only on $p$.
\end{thm}
Bourgain's main result in \cite{B}, Theorem \ref{thm:main}, though stated and proved for real-valued functions, applies to the complex exponential system $(e^{2\pi i k \theta})_{k=1}^n$ by taking real and imaginary parts. See Remark 5.10 of \cite{DJR} for additional examples of complex exponential systems in the context of Fourier restriction to which Theorem \ref{thm:main} applies.

We also refer the reader to the paper \cite{T1} of Talagrand, where Theorem \ref{thm:main} is proved in a more general setting. Unlike Bourgain's proof, Talagrand's approach avoids using the special properties of the function $x\mapsto |x|^p$ appearing in the definition of the $L^p$ norm. As a result, his argument applies to a broader class of Banach spaces beyond just $L^p$ spaces.

\begin{rem}
The size of $S$ in the above theorem is maximal, in the sense that if \eqref{eq:lambdapprop} holds for all $(a_i)_{i=1}^n$ then $|S|\lesssim n^{2/p}$. To see that this is indeed the case, take $S\subset [n]$ and consider $f(\theta)=\sum_{k\in S} e^{2\pi i k \theta}$. Letting $D_n(\theta)=\sum_{|k|\le n} e^{2\pi i k \theta}$ be the classicial Dirichlet kernel we can use H\"older's inequality to write
\begin{align*}
|S|=f(0)=f\ast D_n(0)\le \|f\|_{p}\|D_n\|_{p'}.
\end{align*}
Using the well known estimate $\|D_n\|_{p'}\lesssim n^{1/p}$ and \eqref{eq:lambdapprop} we get
$$
|S|\lesssim |S|^{1/2}n^{1/p},
$$
which is equivalent to $|S|\lesssim n^{2/p}.$
\end{rem}
The following result, shows that the proof of Theorem \ref{thm:main} reduces to proving Theorem \ref{thm:main2}. 
\begin{prop}
Theorem \ref{thm:main2} implies Theorem \ref{thm:main}.
\end{prop}

\begin{proof}
For each $k\in\N$ consider the system $\Phi_k=\{e^{2\pi i j\theta}: 2^k\le j<2^{k+1}\}$ on the 1-torus $\mathbb T$. By Theorem \ref{thm:main2} (applied with $n=2^k$) there exists a set $S_k\subset[2^k, 2^{k+1})$ such that $S_k=\lfloor 4^{k/p}\rfloor$
and 
\begin{align}\label{eq:Sklambdap}
\| \sum_{j\in S_k}a_je^{2\pi i j \theta} \|_{L^p(\T)}\lesssim \big(\sum_{j\in S_k}|a_j|^2\big)^{1/2}.
\end{align}
Letting $S=\bigcup_{k=1}^\infty S_k$ we obtain by the Littlewood--Paley inequality
\begin{align*}
\| \sum_{j\in S}a_je^{2\pi i j \theta} \|_{L^p(\T)}\lesssim \bigg\|\Big(\sum_{k=1}^\infty\big| \sum_{j\in S_k}a_je^{2\pi i j \theta}\big|^2\Big)^{1/2}\bigg\|_{L^p(\T)}.
\end{align*}
Next, using Minkowski's inequality and then invoking the $\Lambda(p)$-property \eqref{eq:Sklambdap} of each $S_k$ we get
\begin{align*}
 \bigg\|\Big(\sum_{k=1}^\infty\big| \sum_{j\in S_k}a_je^{2\pi i j \theta}\big|^2\Big)^{1/2}\bigg\|_{L^p(\T)}
\lesssim
 \Big(\sum_{k=1}^\infty\big\| \sum_{j\in S_k}a_je^{2\pi i j \theta}\big\|_{L^p(\T)}^2\Big)^{1/2}\lesssim  \big(\sum_{i\in S}|a_i|^2\big)^{1/2}.
\end{align*}
Thus, $S$ is a $\Lambda(p)$-set.

On the other hand, taking 
$$
a_j=\ind{S_k}(j), \qquad j\in\Z,
$$ 
and restricting the region of integration to the range where constructive interference holds, we obtain
\begin{align*}
\| \sum_{j\in S}a_j e^{2\pi i j \theta} \|_{L^r(\T)}&=\| \sum_{j\in S_k}e^{2\pi i j \theta} \|_{L^r(\T)}\ge \Big(\int_{-2^{-k}/10}^{2^{-k}/10} \big| \sum_{j\in S_k}a_je^{2\pi i j \theta}\big|^r\, d\theta\Big)^{1/r}\gtrsim 2^{-k/r}|S_k|
\\
&=2^{k(1/p-1/r)}|S_k|^{1/2}=2^{k(1/p-1/r)}\big(\sum_{i\in S}|a_i|^2\big)^{1/2}. 
\end{align*}
Letting $k\rightarrow \infty$ we see that $S$ is not a $\Lambda(r)$-set for $r>p$.
\end{proof}
In the rest of the paper we will present the proof of Theorem \ref{thm:main2}. For clarity, we will show detailed arguments primarily for the case \( 2 < p \leq 3 \). The cases \( 3 < p < 4 \) and \( p \geq 4 \), which Bourgain treated separately in many arguments, require some technical changes but follow essentially the same ideas. We will outline the approach for these cases and provide detailed explanations only where the arguments significantly differ from the case \( 2 < p \leq 3 \).

The paper is organized as follows. Section \ref{sec:notation} introduces the notation used throughout the article. In Section \ref{sec:dec}, we present the proof of Theorem \ref{thm:main2}, assuming the validity of a key estimate, whose proof is deferred to Section \ref{sec:end}. To prepare for this, we establish a probabilistic inequality in Section \ref{sec:prob} and discuss certain entropy estimates in Section \ref{sec:entropy}.

\subsection{Acknowledgement} We gratefully acknowledge the organizers of the AIM Research Community on Fourier Restriction Theory, Dominique Maldague, Yumeng Ou, Po-Lam Yung, and Ruixiang Zhang, for providing the opportunity that led to our collaboration on this project. This work was started in collaboration with Zirui Zhou. We thank her for her contributions at
the early stages of our work. Finally, we thank the anonymous referee for the valuable suggestions, which helped improve the quality of the paper.

\section{Notation}\label{sec:notation}
In this section we set up our notation that will be used throughout the paper. 

 For a random variable $X$ on a probability space, we will write 
$$
\mathbb{E} X = \mathbb E_\omega X(\omega) = \int_\Omega X(\omega)\,d\omega.
$$

For a measure space $(\Sigma,\mu)$, and real-valued functions $f,g\in L^2(\Sigma,\mu)$ we denote the standard $L^2(d\mu)$ inner product by
$$
\langle f, g\rangle = \int_\Sigma f(u)g(u)\,d\mu(u).
$$


We will frequently use the inequality 
\begin{equation}\label{eq:sumprod}
A+B\le 2AB
\end{equation} 
valid for $A,B\ge 1$.

For vectors $\mathbf{a},\bm{a}\in\R^n$ we will always write $\mathbf{a}=(a_1, \dots, a_n)$ and $\bm{a}=(a_1, \dots, a_n)$. Using the same symbol for components of $\mathbf{a}$ and $\bm{a}$ should not cause any confusion. For $k\in\Z_+$ we denote $[k]=\{1,2,\dots,k\}.$

We denote by $\ind{E}$ the indicator function of a set $E$, and use $|E|$ to represent either the number of elements in $E$ if it is finite, or the Lebesgue measure of $E$ if it is a measurable subset of $\mathbb{R}^n$. When $E \neq \emptyset$, $\mathbb{E}_{m \in E} f(m)$ refers to the average value of a function $f \colon E \to \C$, taken with respect to the appropriate measure.


For two nonnegative quantities $A, B$ we write $A \lesssim B$ if there
is an absolute constant $C \in \R_+$ such that $A\le CB$. If we want to emphasize that $C$ depends on a parameter $\alpha$, then we write $A \lesssim_{\alpha} B$. We will write $A\simeq B$ if $A\lesssim B$ and $B\lesssim A$.

If $(X,d)$ is a bounded metric space and $\delta>0$, then the metric entropy $N_d(X,\delta)$ with respect to the metric $d$ is the minimum number of $d$-balls of radius $\delta$ needed to cover $X$.

\section{Setting up the proof of Theorem \ref{thm:main2}}\label{sec:dec}
In this section we begin discussing the proof of  Theorem \ref{thm:main2}. 
\subsection{Overview and beginning of the proof of Theorem \ref{thm:main2}}

Let $\Phi = \big(\varphi_i(u)\big)_{i=1}^n$ be a 1-bounded system of real-valued functions defined on a measure space $(\Sigma, \mu)$ with total measure $\mu(\Sigma) = 1$, orthogonal under the $L^2(d\mu)$ inner product defined by
\[
\langle f,g\rangle := \int_\Sigma f(u)g(u)\,d\mu(u).
\]
In what follows, we will suppress all mention of the space $\Sigma$ and the measure $\mu$. We will follow Bourgain's notational convention and write the integral $\int f(u)\,du$ and the $L^p$-norm $\|f\|_{L^p(du)}$ instead of $\int_\Sigma f(u)\,d\mu(u)$ and $\|f\|_{L^p(\Sigma,\mu)}$, respectively.

For $S\subset [n]$ define  
\begin{equation}\label{eq:KS}
K_S = \sup_{|\mathbf{a}|\le 1}\|\sum_{i\in S}a_i\varphi_i\|_{L^p(du)},
\end{equation}
where we use the notation $\mathbf a = (a_1,\dots,a_n)$, and $|\mathbf{a}|$ for the Euclidean norm. We suppress $\Sigma$ and the measure $\mu$ from the notation, since the bounds we obtain will not depend on the particular probability space $(\Sigma,\mu)$ we begin from.

To prove Theorem \ref{thm:main2} we need to show that there exists $S$ with $|S|\simeq n^{2/p}$ such that $K_S\lesssim 1$. This will be accomplished by showing that generic random sets $S$ of the desired size satisfy $K_S\lesssim 1$, see Theorem \ref{thm:generic-lambda(p)} below. We will set this up appropriately, starting by describing a general decomposition of a function $\sum_{i\in S}a_i\varphi_i$ with $|\mathbf{a}|= 1$ 
followed by a discussion of the randomization of $S$.

We will need the following simple numerical lemma.
\begin{lem}\label{numerical23}
Let $x,y\in\R$. If $2< p \le 3$, then
\begin{equation}\label{eq:numerical1}
|x + y|^p \le (x + y)^2|y|^{p-2} + (1+|x|)^p + 2x(1+|x|)^{p-2}y + (1+|x|)^{p-2}y^2.
\end{equation}
Moreover, if $p>3$, then there exists an absolute constant $C>0$ such that
\begin{equation}\label{eq:numerical2}
|x + y|^p \le |x + y|^{p-2}|x|^{2} + C(|x|+|y|)^{p-3}|y|^3 + 2x|x|^{p-2}y + (2p-3)|x|^{p-2}y^2.
\end{equation}
\end{lem}
\begin{proof}
To prove \eqref{eq:numerical1} it suffices to write
\[
|x+y|^p \le (x+y)^2(|x|^{p-2}+|y|^{p-2}) \le (x+y)^2|y|^{p-2} +(x+y)^2(1+|x|)^{p-2}.
\]

For \eqref{eq:numerical2} we begin with splitting
$$
|x+y|^p=|x+y|^{p-2}x^2+|x+y|^{p-2}(2xy+y^2).
$$
Now, it remains to note that the second term on the right-hand side above can be bounded with the aid of the inequality
$$
\big| |x+y|^{p-2}-|x|^{p-2} -(p-2)|x|^{p-4}xy\big|\lesssim (|x|+|y|)^{p-4} y^2.
$$
That completes the proof of the lemma.
\end{proof}

The first step in the proof of Theorem \ref{thm:main2} consists of an ingenious decomposition of the generic sequence $(a_i)$. This allows us to reduce the proof to estimating suitable linearized expressions coming from Lemma \ref{numerical23}.

\begin{prop}\label{prop:bootstrap0}
Let $S\subset[n]$. For any $p>2$ the following estimate holds  
\begin{align*}
K_S^p \lesssim_p & 1+K_S^{p-1} + \sup_{(\mathbf{a},\mathbf{b},I)\in \mathcal{A}(S)}\Big[|\langle \sum_{i\in S}b_i\varphi_i, \sum_{i\in I\cap S}a_i\varphi_i(1+|\sum_{i\in I\cap S}a_i\varphi_i|)^{p-2}\rangle|
\\
&\quad+|\langle \sum_{i\in S}b_i\varphi_i, \sum_{i\in S}b_i\varphi_i(1+|\sum_{i\in I\cap S}a_i\varphi_i|)^{p-2}\rangle|\Big].
\end{align*}
where 
$$
\mathcal{A}(S):=\{(\mathbf{a},\mathbf{b},I): I\subset [n], \mathbf{a} = (a_i)_{i\in I\cap S} , \mathbf{b}=(b_i)_{i\in S}, |\mathbf{a}|,|\mathbf{b}|\le 1, \max_i|b_i| \le |I|^{-1/2}\}.
$$ 
\end{prop}
\begin{rem}\label{rem:supp}
Note if $(\mathbf{a},\mathbf{b},I)\in \mathcal{A}(S)$ then the sequences $\mathbf{a}$ and $\mathbf{b}$ are both supported in $S$.
\end{rem}

\begin{proof}[Proof of Proposition \ref{prop:bootstrap0}]
We will only present the proof in the case $p\in(2,3]$. The argument for handling the case $p>3$ is similar, with the only differences being that \eqref{eq:numerical2} is applied instead of \eqref{eq:numerical1}, and condition \eqref{eq:gammacond1} below needs to be replaced by 
\begin{equation}\label{eq:gammacond2}
1-\gamma^2 + C\gamma^3 < 1,
\end{equation}
where $C>0$ is the constant from \eqref{eq:numerical2}.

 Choose $0 < \gamma < 1$ satisfying
\begin{equation}\label{eq:gammacond1}
(1-\gamma^2)^{(p-2)/2} + \gamma^p < 1.
\end{equation}
This is possible since if $\gamma$ is sufficiently small, then by Taylor's theorem
$$
(1-\gamma^2)^{(p-2)/2} + \gamma^p=1-\frac{p-2}{2}\gamma^2 + \gamma^p +O(\gamma^4),
$$ 
and the expression on the right-hand side is smaller than 1, since $p > 2$.
Fix now $\mathbf{a} = (a_i)_{i\in S}$ satisfying $|\mathbf{a}| = 1$. By letting $a_i=0$ for $i\in [n]\setminus S$ we may assume that $\mathbf{a}$ is a vector with $n$ components $\mathbf{a} = (a_i)_{i=1}^n$ such that $|\supp\mathbf{a}| \le |S|$, and moreover, the components of $\mathbf{a}$ are arranged in decreasing order of magnitude
\[
|a_1| \ge |a_2| \ge \dots \ge |a_{|S|}|\ge |a_{|S|+1}|=\dots=|a_n|=0.
\]
Define 
$$
m_0 =
\begin{cases} \max\{m : \sum_{i=1}^m|a_i|^2 < \gamma^2\}, \quad &\textrm{if}\quad |a_1|<\gamma^2,
\\
0, \quad &\textrm{if}\quad |a_1|\ge\gamma^2.
\end{cases}
$$ Since $|\mathbf{a}| = 1$ it follows that $m_0 < |S|$.  Moreover,
\[
\sum_{i=m_0+2}^{|S|}|a_i|^2 \le 1 - \gamma^2.
\]
To see this, note that by the maximality of $m_0$,
\[
\sum_{i=m_0+2}^{|S|}|a_i|^2 = 1-\sum_{i=1}^{m_0+1}|a_i|^2 \le 1-\gamma^2.
\]
First assume that $m_0\ge1$. We will comment on handling the case $m_0=0$ later.  Define subsets $I=I_{\mathbf{a}}$ and $J=J_{\mathbf{a}}$ of $[n]$ by
\begin{equation}
I = \{1,\dots,m_0\}\quad\text{and}\quad J = \{m_0+2,\dots,n\}.
\end{equation}
Then 
\[
\min_{i\in I}|a_i|\ge \max_{j\in J}|a_j|,
\]
and
\begin{equation}\label{eq:gamma_split}
\sum_{i\in I}|a_i|^2 < \gamma^2,\quad \text{and}\quad \sum_{j\in J}|a_j|^2 \le 1-\gamma^2.
\end{equation}
Furthermore,
\[
1 \ge \sum_{i=1}^{m_0}|a_i|^2 \ge m_0|a_{m_0+1}|^2\ge m_0 \max_{j\in J}|a_j|^2,
\]
so that $\max_{i\in J}|a_i|\le |I|^{-1/2}$. In anticipation of the application of Lemma \ref{numerical23}, let
\[
x(u) = \sum_{i\in I}a_i\varphi_i(u),\quad y(u) = \sum_{j\in J}a_j\varphi_j(u).
\]
Writing $\sum_{i\in S}a_i\varphi_i = x + y + a_{m_0+1}\varphi_{m_0+1}$, using the $1$-boundedness assumption to estimate $\|a_{m_0+1}\varphi_{m_0+1}\|_\infty \le 1$ and applying H\"older's inequality, we get
\begin{align*}
\|\sum_{i\in S}a_i\varphi_i\|_p^p&\le \int (1+|x(u)+y(u)|)^p\,du
\\
&\le \int |x(u)+y(u)|^p\,du + p\int (1+|x(u)+y(u)|)^{p-1}\,du
\\
&\le 
\int |x(u)+y(u)|^p\,du + p2^{p-1}\big(1+\int |x(u)+y(u)|^{p-1}\,du\big)
\\
&\le
\int |x(u)+y(u)|^p\,du + p2^{p-1}\big(1+(\int |x(u)+y(u)|^p\,du)^{(p-1)/p}\big)
\\
&=\|x+y\|_p^p + C_p(1+\|x+y\|_p^{p-1})\le \|x+y\|_p^p + C_p(1+K_S^{p-1}).
\end{align*}
Note that in the case $m_0=0$ one has $x\equiv 0$ and the above argument gives
\begin{align*}
\|\sum_{i\in S}a_i\varphi_i\|_p^p
\le \|y\|_p^p + C_p(1+K_S^{p-1}).
\end{align*}
Then, by the definition of $K_S$, we have, still in the case $m_0=0$,
\begin{align}\nonumber
\|y\|_p^p + C_p(1+K_S^{p-1})&\le K_S^p(\sum_{j\in J}|a_j|^2)^{p/2} + C_p(1+K_S^{p-1})
\\ \label{eq:m0zero}
&\le K_S^p(1-\gamma^2)^{p/2} + C_p(1+K_S^{p-1}).
\end{align}
Returning to the case $m_0\ge 1$, we use  Lemma \ref{numerical23} and H\"older's inequality to get
\begin{align*}
\|x+y\|_p^p &\le \|x+y\|_p^2\|y\|_p^{p-2} + \|1+|x|\|_p^p + 2|\langle y,x(1+|x|)^{p-2}\rangle| + |\langle y,y(1+|x|)^{p-2}\rangle|.
\end{align*}
By the definition of $K_S$, the first term is bounded by $K_S^p(\sum_{j\in J}|a_j|^2)^{(p-2)/2}$. To estimate the second term, we argue as above getting
\begin{align*}
\|1+|x|\|_p^p\le \|x\|_p^p + C_p(1+K_S^{p-1})\le K_S^p(\sum_{i\in I}|a_i|^2)^{p/2}+ C_p(1+K_S^{p-1}).
\end{align*}
Combining the above estimates with \eqref{eq:gamma_split} we obtain for $m_0\ge 1$
\begin{align}\nonumber
\|\sum_{i\in S}a_i\varphi_i\|_p^p &\le [(1-\gamma^2)^{(p-2)/2} + \gamma^p]K_S^p + 2|\langle y, x(1+|x|)^{p-2}\rangle|+|\langle y, y(1+|x|)^{p-2}\rangle|
\\ \label{eq:m0nonzero}
&\quad+ C_p(1+K_S^{p-1}).
\end{align}
Taking into account \eqref{eq:m0zero} and \eqref{eq:m0nonzero}, and taking the supremum over sequences $|\mathbf{a}|\le 1$ we get
\begin{align*}
K_S^p &\le [(1-\gamma^2)^{(p-2)/2} + \gamma^p]K_S^p + \sup_{(\mathbf{a},\mathbf{b},I)\in \mathcal{A}(S)}\Big[2|\langle \sum_{i\in S}b_i\varphi_i, \sum_{i\in I\cap S}a_i\varphi_i(1+|\sum_{i\in I\cap S}a_i\varphi_i|)^{p-2}\rangle|
\\
&\quad+|\langle \sum_{i\in S}b_i\varphi_i, \sum_{i\in S}b_i\varphi_i(1+|\sum_{i\in I\cap S}a_i\varphi_i|)^{p-2}\rangle|\Big] + C_p(1+K_S^{p-1}),
\end{align*}
where the supremum on the right-hand side is taken over the set 
$$
\mathcal{A}(S)=\{(\mathbf{a},\mathbf{b},I): I\subset [n], \mathbf{a} = (a_i)_{i\in I\cap S} , \mathbf{b}=(b_i)_{i\in S}, |\mathbf{a}|,|\mathbf{b}|\le 1, \max_i|b_i| \le |I|^{-1/2}\}.
$$ 
Note in particular we do not require that the vectors $\mathbf{a}$ and $\mathbf{b}$ have disjoint supports, unlike in the decomposition we made using the disjoint sets $I$ and $J$. Using the condition $(1-\gamma^2)^{(p-2)/2} + \gamma^p < 1$ we see that the first term on the right-hand side can be absorbed by the left-hand side, giving
\begin{align*}
K_S^p \lesssim_p & 1+K_S^{p-1} + \sup_{(\mathbf{a},\mathbf{b},I)\in \mathcal{A}(S)}\Big[|\langle \sum_{i\in S}b_i\varphi_i, \sum_{i\in I\cap S}a_i\varphi_i(1+|\sum_{i\in I\cap S}a_i\varphi_i|)^{p-2}\rangle|
\\
&\quad+|\langle \sum_{i\in S}b_i\varphi_i, \sum_{i\in S}b_i\varphi_i(1+|\sum_{i\in I\cap S}a_i\varphi_i|)^{p-2}\rangle|\Big].
\end{align*}
That concludes the proof.
\end{proof}

\subsection{Decoupling}
To further decompose the expressions arising from Proposition \ref{prop:bootstrap0} we will need an important probabilistic decoupling lemma. In the proof of the probabilistic decoupling lemma, we will use the following generalization of Khintchine's inequality, see \cite[Theorem 2, Section 10.3]{CT}.

\begin{lem}[Marcinkiewicz--Zygmund inequality] \label{lem:MZ}
Let $1\le p <\infty$ and let $(X_i)_{i=1}^n$ be a family of independent random variables with $\mathbb{E}(X_i)=0$ and such that $\mathbb{E}(|X_i|^p)<\infty$. Then
\begin{equation*}
\mathbb{E}\Big(\big|\sum_{i=1}^n X_i\big|^p\Big) \simeq \mathbb{E}\Big(\big(\sum_{i=1}^n |X_i|^2\big)^{p/2}\Big),
\end{equation*}
with an implicit constant that depends only on $p$.
\end{lem}

The probabilistic decoupling lemma reads as follows.
\begin{lem}\label{lem:dec}
Consider for $\alpha\in[3]$ functions $\phi_\alpha\colon\R\rightarrow\R$, satisfying
\begin{equation}\label{ineq:growth}
|\phi_\alpha(x)| \le C(1+|x|)^{p_\alpha},
\end{equation}
\begin{equation}\label{ineq:smoothness}
|\phi_\alpha(x)-\phi_\alpha(y)| \le C(1+|x|+|y|)^{p_\alpha-\delta}|x-y|^\delta,
\end{equation}
where $p_\alpha\ge \delta > 0$.

Let $\mathbf{u}=(u_i)_{i=1}^n,\mathbf{v}=(v_i)_{i=1}^n,\mathbf{w}=(w_i)_{i=1}^n$ be vectors in $\R^n$ with $|\mathbf{u}|,|\mathbf{v}|,|\mathbf{w}|\le 1$ and let $\{\eta_i\}_{i=1}^n,\{\zeta_i\}_{i=1}^n$ be independent $\{0,1\}$-valued random variables of respective means
\begin{equation*}
\int\eta_i(t)\,dt = \frac13 \quad\text{and}\quad \int\zeta_i(t)\,dt = \frac12, \qquad 1\le i\le n.
\end{equation*}
Define the disjoint random sets
\begin{align*}
R^1_t &= \{1\le i \le n: \eta_i(t) = 1\},\\
R^2_t &= \{1\le i\le n : \eta_i(t) = 0, \zeta_i(t) = 1\},\\ 
R^3_t &= \{1\le i\le n : \eta_i(t) = 0, \zeta_i(t) = 0\}.
\end{align*}
Then
\begin{align}
\nonumber
&\left| \int\phi_1\Big(\sum_{i\in R^1_t}u_i\Big)\phi_2\Big(\sum_{i\in R^2_t}v_i\Big)\phi_3\Big(\sum_{i\in R^3_t}w_i\Big)\,dt - \phi_1\Big(\frac13\sum_{i=1}^n u_i\Big)\phi_2\Big(\frac13\sum_{i=1}^n v_i\Big)\phi_3\Big(\frac13\sum_{i=1}^n w_i\Big)\right|\\
\label{ineq:decoupling}
&\qquad\qquad\qquad\qquad\le C\Big(1+\Big|\sum_{i=1}^n u_i\Big| + \Big|\sum_{i=1}^n v_i\Big| + \Big|\sum_{i=1}^n w_i\Big|\Big)^{p-\delta},
\end{align}
where $p = p_1 + p_2 + p_3$.
\end{lem}
\begin{proof}
Let $U(t) = \sum_{i\in R^1_t}u_i$, $V(t) = \sum_{i\in R^2_t}v_i$, $W(t) = \sum_{i\in R^3_t}w_i$. Then \eqref{ineq:decoupling} is equivalent to 
$$
\Big| \mathbb E\big(\phi_1(U)\phi_2(V)\phi_3(W)\big)-\phi_1(\mathbb EU)\phi_2(\mathbb EV)\phi_3(\mathbb EW)\Big|\lesssim(1+|\mathbb EU|+|\mathbb EV|+|\mathbb EW|)^{p-\delta}.
$$
To prove the above estimate we begin by splitting the left-hand side as follows
\begin{align}\label{eq:expr}
&\Big| \mathbb E\big(\phi_1(U)\phi_2(V)\phi_3(W)\big)-\phi_1(\mathbb EU)\phi_2(\mathbb EV)\phi_3(\mathbb EW)\Big|
\\ \nonumber
&\le 
\Big|\mathbb E[\big(\phi_1(U)-\phi_1(\mathbb EU)\big)\phi_2(V)\phi_3(W)]\Big| 
+ \Big|\phi_1(\mathbb EU)\mathbb E[\big(\phi_2(V)-\phi_2(\mathbb EV)\big)\phi_3(W)]\Big|
\\ \nonumber
&\quad+ \Big|\phi_1(\mathbb EU)\phi_2(\mathbb EV)\mathbb E\big[\phi_3(W)-\phi_3(\mathbb EW)]\Big|.
\end{align}
Each of the three terms on the right-hand side of \eqref{eq:expr} can be treated similarly, so we only provide the details required to estimate the first term.  By the smoothness assumption \eqref{ineq:smoothness} we have
\begin{align*}
|\phi_1(U)-\phi_1(\mathbb EU)| 
&\lesssim (1+|U|+|\mathbb EU|)^{p_1-\delta}|U-\mathbb EU|^\delta
\\
&\lesssim (1+|\mathbb EU|+|U-\mathbb EU|)^{p_1-\delta}|U-\mathbb EU|^\delta
\\
&\lesssim (1+|\mathbb EU|)^{p_1-\delta}(1+|U-\mathbb EU|)^{p_1-\delta}|U-\mathbb EU|^\delta
\\
&\le (1+|\mathbb EU|)^{p_1-\delta}(1+|U-\mathbb EU|)^{p_1}.
\end{align*}
In the second-to-last estimate  we used \eqref{eq:sumprod}.

On the other hand, by the growth condition \eqref{ineq:growth} applied to the function $\phi_2$ and \eqref{eq:sumprod} we have
\[
|\phi_2(V)| \lesssim (1+|V|)^{p_2}\le(1+|\mathbb EV|)^{p_2} (1+|V-\mathbb EV|)^{p_2}.
\]
Analogously, we get
\[
|\phi_3(W)| \lesssim (1+|W|)^{p_3}\le(1+|\mathbb EW|)^{p_3} (1+|W-\mathbb EW|)^{p_3}.
\]
Combining the above bounds we obtain
\begin{align*}
&\Big|\mathbb E[\big(\phi_1(U)-\phi_1(\mathbb EU)\big)\phi_2(V)\phi_3(W)]\Big|
\\ 
&\quad\lesssim (1+|\mathbb EU|)^{p_1-\delta}(1+|\mathbb EV|)^{p_2}(1+|\mathbb EW|)^{p_3}
\\
&\qquad\times \mathbb E[(1+|U-\mathbb EU|)^{p_1}(1+|V-\mathbb EV|)^{p_2}(1+|W-\mathbb EW|)^{p_3}]
\\
&\quad\le
(1+|\mathbb EU|+|\mathbb EV|+|\mathbb EW|)^{p-\delta}
\\
&\qquad\times\mathbb E[(1+|U-\mathbb EU|)^{p_1}(1+|V-\mathbb EV|)^{p_2}(1+|W-\mathbb EW|)^{p_3}].
\end{align*}
It remains to show that 
$$
\mathbb E[(1+|U-\mathbb EU|)^{p_1}(1+|V-\mathbb EV|)^{p_2}(1+|W-\mathbb EW|)^{p_3}]\lesssim 1.
$$
We have
\begin{align*}
&\mathbb E[(1+|U-\mathbb EU|)^{p_1}(1+|V-\mathbb EV|)^{p_2}(1+|W-\mathbb EW|)^{p_3}]
\\
&\quad\le
\mathbb E[(1+|U-\mathbb EU|+|V-\mathbb EV|+|W-\mathbb EW|)^p]
\\
&\quad\lesssim_p 1+\mathbb{E}\Big( |U-\mathbb EU|^p\Big) + \mathbb{E}\Big( |V-\mathbb EV|^p\Big)+\mathbb{E}\Big( |W-\mathbb EW|^p\Big).
\end{align*}
Then appealing to Lemma \ref{lem:MZ} we get
\begin{align*}
 \mathbb{E}\Big( |U-\mathbb EU|^{p}\Big) &=  \mathbb{E} \Big(|\sum_{i=1}^n(\eta_i^1-\frac13)u_i|^{p}\Big)\lesssim \mathbb{E}\Big(\big(\sum_{i=1}^n |(\eta_i^1-\frac13)u_i|^2\big)^{p/2}\Big)\le \big(\sum_{i=1}^n |u_i|^2\big)^{p/2}
\lesssim 1
\end{align*}
and the terms corresponding to $V$ and $W$ are estimated in the same way. Consequently,
\begin{align*}
\Big|\mathbb E[\big(\phi_1(U)-\phi_1(\mathbb EU)\big)\phi_2(V)\phi_3(W)]\Big|\lesssim (1+|\mathbb EU|+|\mathbb EV|+|\mathbb EW|)^{p-\delta}.
\end{align*}
The remaining two terms on the right-hand side of \eqref{eq:expr} are estimated in the same way. The proof of the lemma is now complete.
\end{proof}

In the next proposition we will apply Lemma \ref{lem:dec} to the linearized expressions arising from Proposition \ref{prop:bootstrap0}.

\begin{prop}\label{prop:bootstrap}
Let $S\subset[n]$. For any $p>2$ the following estimate holds  
\begin{align*}
K_S^p &\lesssim_p 1+K_S^{p-1}+ \int \sup_{(\mathbf{a},\mathbf{b},I)\in \mathcal{A}(S)}|\langle\sum_{i\in S\cap R^1_t}b_i\varphi_i,(\sum_{i\in I\cap S\cap R^2_t}a_i\varphi_i)(1+|\sum_{i\in I\cap S\cap R^3_t}a_i\varphi_i|)^{p-2}\rangle|\,dt
\\
&\quad+\int\sup_{(\mathbf{a},\mathbf{b},I)\in \mathcal{A}(S)}|\langle\sum_{i\in S\cap R^1_t}b_i\varphi_i,(\sum_{i\in S\cap R^2_t}b_i\varphi_i)(1+|\sum_{i\in I\cap S\cap R^3_t}a_i\varphi_i|)^{p-2}\rangle|\,dt,
\end{align*}
where 
$$
\mathcal{A}(S):=\{(\mathbf{a},\mathbf{b},I): I\subset [n], \mathbf{a} = (a_i)_{i\in I\cap S} , \mathbf{b}=(b_i)_{i\in S}, |\mathbf{a}|,|\mathbf{b}|\le 1, \max_i|b_i| \le |I|^{-1/2}\}
$$ 
and $R_t^1, R_t^2$ and $R_t^3$ are the sets defined in Lemma \ref{lem:dec}.
\end{prop}

\begin{proof}
By Proposition \ref{prop:bootstrap0} we have
\begin{align*}
K_S^p \lesssim_p & 1+K_S^{p-1} + \sup_{(\mathbf{a},\mathbf{b},I)\in \mathcal{A}(S)}\Big[|\langle \sum_{i\in S}b_i\varphi_i, \sum_{i\in I\cap S}a_i\varphi_i(1+|\sum_{i\in I\cap S}a_i\varphi_i|)^{p-2}\rangle|
\\
&\quad+|\langle \sum_{i\in S}b_i\varphi_i, \sum_{i\in S}b_i\varphi_i(1+|\sum_{i\in I\cap S}a_i\varphi_i|)^{p-2}\rangle|\Big].
\end{align*}
We concentrate on estimating the first term within the supremum, as the analysis of the second term follows similarly and can be left to the reader for verification.

The argument relies on applying the probabilistic decoupling Lemma \ref{lem:dec} pointwise within the integration over 
$u$. We apply Lemma \ref{lem:dec} to the functions $\phi_1(x) = \phi_2(x) = x$, and $\phi_3(x) = (1/3+|x|)^{p-2}$ with $p_1 = p_2 = 1$, and $p_3 = \delta = p-2,$ $\mathbf{u}=(b_i\varphi_i(u)\mathbf{1}_S(i))_{i=1}^n$, and $\mathbf{v}=\mathbf{w}=(a_i\varphi_i(u)\mathbf{1}_{I\cap S}(i))_{i=1}^n$.  With $R^1_t,R^2_t,R^3_t$ the random subsets of $[n]$ as in Lemma \ref{lem:dec}, we have a pointwise estimate for any $u$
\begin{align*}
&\bigg|\mathbb \int \big(\sum_{i\in S\cap R^1_t}b_i\varphi_i(u)\big)\big(\sum_{i\in I\cap S\cap R^2_t}a_i\varphi_i(u)\big)\big(1/3+|\sum_{i\in I\cap S\cap R^3_t}a_i\varphi_i(u)|\big)^{p-2}\, dt
\\
&\quad- 3^{-p}\sum_{i\in S}b_i\varphi_i(u) \sum_{i\in I\cap S}a_i\varphi_i(u)(1+|\sum_{i\in I\cap S}a_i\varphi_i(u)|)^{p-2}\bigg|
\\
&\quad \lesssim_p (1+|\sum_{i\in I\cap S}a_i\varphi_i(u)|+|\sum_{i\in S}b_i\varphi_i(u)|)^2.
\end{align*}
Integrating in $u$ and using the orthogonality of the system $(\varphi_i)_{i=1}^n$, we see that
\begin{align}\label{314}
&\sup_{(\mathbf{a},\mathbf{b},I)\in \mathcal{A}(S)}|\langle \sum_{i\in S}b_i\varphi_i, \sum_{i\in I\cap S}a_i\varphi_i(1+|\sum_{i\in I\cap S}a_i\varphi_i|)^{p-2}\rangle|
\\ \nonumber
&\quad\lesssim_p 1+ \int\sup_{(\mathbf{a},\mathbf{b},I)\in \mathcal{A}(S)}|\langle\sum_{i\in S\cap R^1_t}b_i\varphi_i,(\sum_{i\in I\cap S\cap R^2_t}a_i\varphi_i)(1+|\sum_{i\in I\cap S\cap R^3_t}a_i\varphi_i|)^{p-2}\rangle| \, dt.
\end{align}
Similar considerations give 
\begin{align*}
&\sup_{(\mathbf{a},\mathbf{b},I)\in \mathcal{A}(S)}|\langle \sum_{i\in S}b_i\varphi_i, \sum_{i\in S}b_i\varphi_i(1+|\sum_{i\in I\cap S}a_i\varphi_i|)^{p-2}\rangle|
\\ \nonumber
&\quad\lesssim_p 1+\int\sup_{(\mathbf{a},\mathbf{b},I)\in \mathcal{A}(S)}|\langle\sum_{i\in S\cap R^1_t}b_i\varphi_i,(\sum_{i\in S\cap R^2_t}b_i\varphi_i)(1+|\sum_{i\in I\cap S\cap R^3_t}a_i\varphi_i|)^{p-2}|\rangle|\, dt,
\end{align*}
which concludes the proof.
\end{proof}

\begin{rem}
    On line (3.20) of Bourgain's paper \cite{B}, there is a minor typographical error. In the sum over $i \in S \cap R_t^2$, the coefficients are given as $a_i$, but they should be $b_i$, consistent with the earlier argument. This has no bearing on the overall argument.
\end{rem}

\subsection{Randomization of the set $S$}
In this subsection we describe the randomization of the set $S$ we mentioned earlier. First we need the following elementary lemma, which may be thought of as another variant of probabilistic decoupling. 
\begin{prop}\label{decoupledomega}
Suppose $\{\xi_i:i\in[n]\}$ are independent and bounded real-valued
random variables defined on some probability space $\Omega$, and $\{A_l\}_{l=1}^k$ are pairwise disjoint subsets of $[n]$. Let $\mathcal A\subset\R^n$. Then
\begin{equation*}
\mathbb E _\omega \sup_{\mathbf{a}\in \mathcal A}\Big |\prod_{l=1}^k\big(\sum_{i\in A_l}a_i\xi_i(\omega)\big)\Big| = 
\mathbb E _{\omega_1, \dots, \omega_k} \sup_{\mathbf{a}\in \mathcal A}\Big |\prod_{l=1}^k\big(\sum_{i\in A_l}a_i\xi_i(\omega_l)\big)\Big |.
\end{equation*}
\end{prop}
\begin{proof}
We will show the proof in the case $k=2$ as the general case follows by induction. For disjoint sets $A_1=A$ and $A_2=B$ 
let $\mathcal L$ be the law of the random vector $(\xi_i)_{i\in A\cup B}$, and let $\mathcal L_A,\mathcal L_B$ be
the laws of the random vectors $(\xi_i)_{i\in A},(\xi_j)_{j\in B}$, respectively. By the definition of
$\mathcal L$ 
\begin{align*}
\mathbb E _\omega& \sup_{\mathbf{a}\in \mathcal A}\Big |\big(\sum_{i\in A}a_i\xi_i(\omega)\big)\big(\sum_{j\in B}a_j\xi_j(\omega)\big)\Big| \\
&=\int_{\mathbb R^{|A|+|B|}}\sup_{\mathbf{a}\in \mathcal A}\Big |\big(\sum_{i\in A}a_ix_i\big)\big(\sum_{j\in B}a_jx_j\big)\Big| \mathcal L\big(\prod_{i\in A}dx_i\prod_{j\in B}dx_j\big)
\end{align*}
By the independence of $\xi_i$, and the disjointness of $A$, $B$, we have the identity 
$$
\mathcal L\big(\prod_{i\in A}dx_i\prod_{j\in B}dx_j\big) = \mathcal L_A\big(\prod_{i\in A}dx_i\big)\mathcal L_B\big(\prod_{j\in B}dx_j\big),
$$ 
which can be written as
$$
\mathcal L = \mathcal L_A\otimes \mathcal L_B.
$$
Expressing the last integral as an iterated integral and using the definition of $\mathcal{L}$, we get the desired identity.
\end{proof}
 What we have said up to this point applies to an arbitrary $S\subset[n]$. Now we will specialize and choose a random set $S=S_\omega$ (note this is an additional independent source of randomness besides the random index sets $R^1_t,R^2_t,R^3_t$). 

We begin with treating the case $p\in(2,4)$, the case $p\ge4$ will require some modification that we will describe later.
\subsection{The case $p\in(2,4)$} 
Let $\{\xi_i(\omega):i\in[n]\}$ be independent $\{0,1\}$-valued random variables (selectors) on some probability space $\Omega$ of mean $\delta = \int_\Omega\xi_i(\omega)\,d\omega$ satisfying
\begin{equation}\label{eq:n0_def}
\delta n = n^{2/p} =: n_0
\end{equation}
and consider the random set
\[
S_\omega = \{i\in[n]:\xi_i(\omega)=1\}, \qquad \omega\in\Omega,
\]
which has expected size $\E_\omega|S_\omega|=\delta n = n^{2/p}$. Observe that $|S_\omega|=\sum_{i=1}^n \xi_i(\omega)$, thus by standard large deviation estimates for binomial random variables, as recorded in Proposition \ref{prop:largedev}, we get that $\frac{1}{10}n^{2/p}<|S_\omega|<10n^{2/p}$ holds with high probability. Denote $K(\omega) = K(S_\omega)$.  We will prove that there exists $C(p)>0$ depending only on $p$ such that 
\begin{equation}\label{eqn:random-lambdap}
    \mathbb E_\omega K(\omega)^p \le C(p).
\end{equation}
The inequality \eqref{eqn:random-lambdap} implies the following theorem about the $\Lambda(p)$-constant of the random sets $S_\omega$.
\begin{thm}\label{thm:generic-lambda(p)}
    Assume that \eqref{eqn:random-lambdap} holds. Then for each $\varepsilon>0$, there is $N(\varepsilon,p)\in\mathbb N$ so that for all $n\ge N(\varepsilon,p)$ there exists an event $E\subset\Omega$ of probability
    \[
    \mathbb P(E) \ge 1 - 2C(p)\varepsilon^p
    \]
    such that $K(\omega)\le \frac{1}{\varepsilon}$ and $\frac{1}{10}n^{2/p}<|S_\omega|<10 n^{2/p}$ for all $\omega \in E$.
\end{thm}
\begin{proof}Set $G = \{\omega\in\Omega:\frac{1}{10}n^{2/p}<|S_\omega|<10 n^{2/p}\}$ and $F = \{\omega\in\Omega: K(\omega) \le \frac{1}{\varepsilon}\}$. 
    By the large deviation estimate of Proposition \ref{prop:largedev},
    \[
    \mathbb P(G) \ge 1-2e^{-\frac{n\delta}{2}}=1-2e^{-\frac{n^{2/p}}{2}}.
    \]
    This estimate ensures that the random set $S_\omega$ has the correct cardinality with high probability, provided $n$ is sufficiently large. 

    On the other hand, by Chebyshev's inequality and \eqref{eqn:random-lambdap} we get
    \[
    \frac{1}{\varepsilon^p}\mathbb P(F^{{\complement}})  \le \mathbb E_\omega K(\omega)^p \le C(p).
    \]
Then 
\[
\mathbb P(F)\ge 1-C(p)\varepsilon^p .
\]

    Finally, set $E := G\cap F$. Then by the inclusion-exclusion principle and the above estimates we obtain
    \[
    \mathbb P(E) \ge \mathbb{P}(F)+ \mathbb{P}(G) -1\ge1-2e^{-\frac{n^{2/p}}{2}}-C(p)\varepsilon^p \ge 1-2C(p)\varepsilon^p,
    \]
    provided that $n\ge N(\varepsilon,p)$ is sufficiently large so that $2e^{-\frac{n^{2/p}}{2}}\le C(p)\varepsilon^p$.
\end{proof}
Theorem \ref{thm:generic-lambda(p)} shows that \eqref{eqn:random-lambdap} in particular implies Theorem \ref{thm:main2}, because it implies the existence of some $\omega_0\in\Omega$ such that $|S_{\omega_0}| \simeq n^{2/p}$ and such that $K(S_{\omega_0})\lesssim 1$. Moreover, it quantifies how the $\Lambda(p)$-property holds for ``most'' random sets $S_\omega$ of cardinality about $n^{2/p}$. 

\begin{rem}\label{rem:largedev}
We claim that in order to prove the key inequality \eqref{eqn:random-lambdap}, it suffices to restrict the integration in $\omega$ to the event $\{\omega\in\Omega:\frac{1}{10}n_0< |S_\omega|\le 10n_0\}$.  To see this, we appeal to 
Proposition \ref{prop:largedev} which shows that the large deviation event
$$
L:=\big\{\omega\in\Omega:|S_\omega|\notin (\frac{1}{10}n_0,10n_0]\big\}=\big\{\omega\in\Omega: \frac{1}{n}\sum_{i=1}^n\xi_i(\omega)\notin (\frac{1}{10}\delta,10\delta]\big\}
$$ has exponentially small probability, that is
\[
\mathbb{P}(L) \le  2e^{- \frac{n \delta}{2}}.
\]
As the system $\Phi=(\varphi_i)_{i=1}^n$ is 1-bounded, one can easily see that the contribution from $\omega$ such that $|S_\omega|>10n_0$ or $|S_\omega|\le\frac{1}{10}n_0$ is at most $n^{C}e^{-n^c}$, which is $O(1)$ provided $n$ is sufficiently large. Therefore, in what follows we can and do restrict integration in $\omega, \omega_1, \omega_2$ and $\omega_3$ to $\Omega\setminus L$.  To avoid cumbersome notation we write $\Omega$ instead of $\Omega\setminus L$.

Finally, we note that for such $\omega$,
\begin{equation}\label{eqn:K-at-least-1}
K(\omega)\ge 1
\end{equation}
since the set $S_\omega$ is nonempty.
\end{rem}

In the remaining part of the paper for $\omega\in\Omega$ and a sequence $\mathbf{d}=(d_1,\dots,d_n)$ we denote
\begin{align*}
f_{\mathbf{d},\omega}=\sum_{i=1}^n\xi_{i}(\omega)d_i\varphi_i.
\end{align*}

We will prove the following.

\begin{prop}\label{prop:dec}
If $p>2$, then
\begin{align*}
\int_\Omega K^p(\omega)\,d\omega&\lesssim 
1+\int_\Omega K^{p-1}(\omega)\,d\omega
\\
&\quad+\int_\Omega\int_\Omega\int_\Omega \sup_{(\bm a,\bm b,\bm c)\in \mathcal{A}_1}\bigg|\langle f_{\bm a,\omega_1}, 	f_{\bm b,\omega_2}(1+|f_{\bm c,\omega_3}|)^{p-2}\rangle\bigg|\,d\omega_1\,d\omega_2\,d\omega_3
\\
&\quad+\int_\Omega\int_\Omega\int_\Omega \sup_{(\bm a,\bm b,\bm c)\in \mathcal{A}_2}\bigg|\langle f_{\bm a,\omega_1}, 	f_{\bm b,\omega_2}(1+|f_{\bm c,\omega_3}|)^{p-2}\rangle\bigg|\,d\omega_1\,d\omega_2\,d\omega_3,
\end{align*}
where the suprema are taken over the sets 
\begin{align*}
\mathcal{A}_1:=\{(\bm a,\bm b,\bm c):\ &|\supp \bm a|, |\supp \bm b|, |\supp \bm c|\le 10n_0, |\bm a|,|\bm b|,|\bm c|\le 1
\\ 
&\quad\text{and}\quad\max_{1\le i\le n}|a_i|\le (|\supp\bm b|+|\supp\bm c|)^{-1/2}\},
\\
\mathcal{A}_2:=\{(\bm a,\bm b,\bm c):\ &|\supp \bm a|, |\supp \bm b|, |\supp \bm c|\le 10n_0, |\bm a|,|\bm b|,|\bm c|\le 1 
\\
&\quad \text{and}\quad \max_{1\le i\le n}(|a_i|,|b_i|)\le |\supp\bm c|^{-1/2}\}.
\end{align*}
\end{prop}

\begin{proof}
Applying Proposition \ref{prop:bootstrap} with $S=S_\omega$ for each $\omega\in\Omega$ and replacing the summation over $S_\omega$ by summation over $[n]$, taking the selectors $\xi_i(\omega)$ into account, we get:
\begin{align*}
\int K^p(\omega)\,d\omega &\lesssim_p 1+\int K^{p-1}(\omega)\,d\omega
\\
&\quad+\iint \sup_{(\mathbf{a},\mathbf{b},I)\in \mathcal{A}(S_\omega)} |\langle \sum_{i\in R_t^1}\xi_i(\omega)b_i\varphi_i, (\sum_{i\in I\cap R_t^2}\xi_i(\omega)a_i\varphi_i) 
\\
&\qquad\qquad\qquad\qquad\qquad\times(1+ |\sum_{i\in I\cap R_t^3}\xi_i(\omega)a_i\varphi_i|)^{p-2}\rangle|\,d\omega\,dt
\\
&\quad+\iint\sup_{(\mathbf{a},\mathbf{b},I)\in \mathcal{A}(S_\omega)}|\langle\sum_{ i\in R^1_t}\xi_i(\omega) b_i\varphi_i,(\sum_{i \in R^2_t}\xi_i(\omega)b_i\varphi_i)
\\
&\qquad\qquad\qquad\qquad\qquad\times(1+|\sum_{i \in I\cap R^3_t}\xi_i(\omega)a_i\varphi_i|)^{p-2}\rangle| \,d\omega\,dt.
\\
&=: 1+\int K^{p-1}(\omega)\,d\omega+J_1 + J_2.
\end{align*}
We will focus on getting a suitable estimate for $J_1$ since the analysis of $J_2$ is similar. We will show that 
$$
J_1\le \iiint \sup_{(\bm a,\bm b,\bm c)\in \mathcal{A}_1}\big|\langle f_{\bm a,\omega_1}, 	f_{\bm b,\omega_2}(1+|f_{\bm c,\omega_3}|)^{p-2}\rangle\big|\,d\omega_1\,d\omega_2\,d\omega_3.
$$

Notice that for each fixed $t$, the index sets $R_t^1,R_t^2,R_t^3$ are disjoint (see Lemma \ref{lem:dec}) so applying Proposition \ref{decoupledomega} (with $k=3$), we obtain
\begin{align*}
J_1= \int\!\!\!\iiint &\sup_{(\mathbf{a},\mathbf{b},I)\in \mathcal{A}(S_\omega)} \big|\langle \sum_{i\in R_t^1}\xi_i(\omega_1)b_i\varphi_i, \sum_{i\in I\cap R_t^2}\xi_i(\omega_2)a_i\varphi_i  (1+ |\sum_{i\in I\cap R_t^3}\xi_i(\omega_3)a_i\varphi_i|)^{p-2}\rangle\big|\\
&\times d\omega_1 \, d\omega_2 \, d\omega_3 \, dt
\end{align*}
Notice that given $(\mathbf{a},\mathbf{b},I)\in \mathcal{A}(S_\omega)$, the vectors
$\bm a := (b_i \mathbf{1}_{R_t^1}(i))_{i=1}^n,
 {\bm b} := (a_i \mathbf{1}_{I\cap R_t^2}(i))_{i=1}^n,
\bm c := (a_i \mathbf{1}_{I\cap R_t^3}(i))_{i=1}^n
$ 
satisfy the defining properties of $\mathcal A_1$:
\begin{itemize}
\item[1)] $|\supp \bm a|, |\supp \bm b|, |\supp \bm c|\le 10n_0$ (since $\supp \mathbf a,\supp \mathbf b \subset S_\omega$ and $|S_\omega|\le 10n_0$, see Remarks \ref{rem:supp} and \ref{rem:largedev}), 
\item[2)] $|\bm a|,|\bm b|,|\bm c|\le 1$,
\item[3)] $\max_i |a_i|  \le (|\supp \bm b| + |\supp \bm c|)^{-1/2}$ (because $|I\cap R_t^2| + |I\cap R_t^3|\le |I|$). (To avoid the possibility of confusion, in this line and what follows, the $a_i$ are the components of $\bm a$.)
\end{itemize}
Hence, for every $\omega$ we can replace the supremum over triples $(\mathbf{a},\mathbf{b},I)\in \mathcal{A}({S_\omega})$ with the supremum over triples of vectors $(\bm a,\bm b,\bm c)\in \mathcal{A}_1$ to get
\begin{align*}
J_1&\le \iiint\sup_{(\bm a,\bm b,\bm c)\in \mathcal{A}_1} |\langle \sum_{i=1}^n \xi_i(\omega_1) a_i\varphi_i, \sum_{i=1}^n \xi_i(\omega_2)b_i\varphi_i (1+|\sum_{i=1}^n \xi_i(\omega_3) c_i\varphi_i|)^{p-2}\rangle|\,d\omega_1\,d\omega_2\,d\omega_3
\\
&=\iiint \sup_{(\bm a,\bm b,\bm c)\in \mathcal{A}_1}|\langle f_{\bm a,\omega_1}, 	f_{\bm b,\omega_2}(1+|f_{\bm c,\omega_3}|)^{p-2}\rangle|\,d\omega_1\,d\omega_2\,d\omega_3.
\end{align*}
Similar arguments show that 
$$
J_2\le \iiint \sup_{(\bm a,\bm b,\bm c)\in \mathcal{A}_2}|\langle f_{\bm a,\omega_1}, 	f_{\bm b,\omega_2}(1+|f_{\bm c,\omega_3}|)^{p-2}\rangle|\,d\omega_1\,d\omega_2\,d\omega_3,
$$
so the proof is finished.
\end{proof}

Let $q_0 = \log n$ and for $1\le m\le n$ define 
\begin{align}\label{eq:Pim}
&		\Pi_{m}=\{\bm a=(a_i)_{i=1}^n:|\bm a|\le 1, |\supp \bm a|\le m\}.		
\end{align}
For fixed $\omega_1, \omega_2$, $\omega_3\in\Omega$ and $m_1,m_2,m_3\in[n]$ we let
\begin{align}\label{eq:K_{m_1,m_2,m_3}}
&	K_{m_1,m_2,m_3}(\omega_1,\omega_2,\omega_3)=\sup_{|A|\le m_1}\sup_{\bm b\in \Pi_{m_2}}\sup_{\bm c\in \Pi_{m_3}}\frac{1}{\sqrt{m_1}}\sum_{i\in A} \xi_i(\omega_1)|\langle \varphi_i, 	f_{\bm b,\omega_2}(1+|f_{\bm c,\omega_3}|)^{p-2}\rangle |.
\end{align}
 In Section \ref{sec:end} we will prove that for $2<p<4$ the following estimate holds for any $\omega_2$ and $\omega_3$
\begin{align}
\label{Kest}
	\|K_{m_1,m_2,m_3}(\omega_1,\omega_2,\omega_3)\|_{L^{q_0}(d\omega_1)}\le (\delta m_3^{\frac{p}{2}-1}+\frac{m_2+m_3}{m_1})^{\frac{1}{2}}(1+K(\omega_2)+K(\omega_3))^{p-\sigma}.
\end{align}
We will now show that \eqref{Kest} implies $\|K(\omega)\|_{L^p(d\omega)}\lesssim 1$.  

Let 
\begin{align} \label{eq:I12def}
    I_j &= \iiint \sup_{(\bm a,\bm b, \bm c)\in \mathcal A_j}\left|\langle f_{\bm a,\omega_1}, f_{\bm b,\omega_2}(1+|f_{\bm c,\omega_3}|)^{p-2}  \rangle \right|\,d\omega_1\,d\omega_2\,d\omega_3, \qquad j=1,2,
\end{align}
be the expressions arising in Proposition \ref{prop:dec}.

We begin with estimating $I_1$.
Note that for any $(\bm a,\bm b,\bm c)\in \mathcal{A}_1$ there exists $m_1 \le 10n_0$ such that $\bm b,\bm c\in \Pi_{m_1}$ and $\max_i|a_i|\le m_1^{-1/2}$. Therefore, for fixed $\omega_2,\omega_3$ we can use the technique of exchanging the supremum with the $L^{q_0}$-norm from Proposition \ref{prop:sup-exchange} to estimate the inner integral in $I_1$ as follows
\begin{align*}
&\int \sup_{(\bm a,\bm b, \bm c)\in \mathcal A_1}\left|\langle f_{\bm a,\omega_1}, f_{\bm b,\omega_2}(1+|f_{\bm c,\omega_3}|)^{p-2}  \rangle \right|\,d\omega_1 
\\
\qquad&\le \int\sup_{1\le m_1\le 10n_0}\sup_{\substack{\bm a\in\Pi_{10n_0},\max_i|a_i|\le m_1^{-1/2}\\\bm b,\bm c\in\Pi_{m_1}}}\left|\langle f_{\bm a,\omega_1}, f_{\bm b,\omega_2}(1+|f_{\bm c,\omega_3}|)^{p-2}  \rangle \right|\,d\omega_1 
\\
\qquad&\le (10n_0)^{1/q_0}\sup_{1\le m_1\le 10n_0}\big\|\sup_{\substack{\bm a\in\Pi_{10n_0},\max_i|a_i|\le m_1^{-1/2}\\\bm b,\bm c\in\Pi_{m_1}}}\left|\langle f_{\bm a,\omega_1}, f_{\bm b,\omega_2}(1+|f_{\bm c,\omega_3}|)^{p-2}  \rangle \right|\big\|_{L^{q_0}(d\omega_1)} 
\\
\qquad&\lesssim \sup_{1\le m_1\le 10n_0}\big\|\sup_{\substack{\bm a\in\Pi_{10n_0},\max_i|a_i|\le m_1^{-1/2}\\\bm b,\bm c\in\Pi_{m_1}}}\left|\langle f_{\bm a,\omega_1}, f_{\bm b,\omega_2}(1+|f_{\bm c,\omega_3}|)^{p-2}  \rangle \right|\big\|_{L^{q_0}(d\omega_1)},
\end{align*}
where in the last inequality we used the fact that
$$ 
(10n_0)^{1/q_0}= e^{q_0^{-1}\log 10n_0}\simeq 1.
$$
Integrating in $\omega_2,\omega_3$ gives
\begin{equation}
\label{firstbound}
I_1\le\iint
\sup_{m_1\le 10n_0}\Big\|\sup_{\substack{\bm a\in\Pi_{10n_0},\max_i|a_i|\le m_1^{-1/2}\\\bm b,\bm c\in\Pi_{m_1}}}\left|\langle f_{\bm a,\omega_1}, f_{\bm b,\omega_2}(1+|f_{\bm c,\omega_3}|)^{p-2}  \rangle \right|\Big\|_{L^{q_0}(d\omega_1)}
d\omega_2\,d\omega_3.
\end{equation}
Since $|\bm a|\le 1$, by Chebyshev's inequality we have the level set estimate $|\{1\le i\le n:m^{-1/2}\le |a_i|\le 2m^{-1/2}\}|\le m$. Therefore, rewriting the sum over $i$ in the definition of $f_{\bm a, \omega_1}$ as a summation over such level sets of $\{|a_i|\}$, we have 
\begin{align*}
I_1&\lesssim \iint \sup_{m_1\le 10n_0} \sum_{\substack{m_1<m\le 10n_0,\\m\ \text{dyadic}}} \Big\|\sup_{\substack{|A|\le m,\\ \bm b, \bm c\in \Pi_{m_1}}} \frac{1}{\sqrt{m}} \sum_{i\in A}\xi_i(\omega_1)   |\langle \varphi_i, f_{\bm b, \omega_2} (1+|f_{\bm c, \omega_3}|)^{p-2}\rangle | \Big\|_{L^{q_0}(d\omega_1)}
\\
&\qquad\qquad\qquad\qquad\qquad\qquad\times
d\omega_2\,d\omega_3 \\
&= \iint \sup_{m_1\le 10n_0}\sum_{\substack{m_1<m\le 10n_0,\\ m\ \text{dyadic}}}\left\|K_{m, m_1, m_1}(\omega_1,\omega_2,\omega_3)\right\|_{L^{q_0}(d\omega_1)}
d\omega_2\,d\omega_3.
\end{align*}
Now by \eqref{Kest}, we have
\begin{align}\nonumber
I_1&\le \Big[ \sup_{m_1\le 10n_0}\sum_{\substack{m_1<m\le 10n_0,\\m\  \text{dyadic}}} \left(\delta m_1^{p/2 - 1}+\frac{m_1}{m}\right)^{1/2}\Big]\|K(\omega)\|_{L^p(d\omega)}^{p-\sigma}
\\ \label{eqn:Kestcompute}
&\le [1+(\delta n_0^{p/2 - 1})]^{1/2}\|K(\omega)\|_{L^p(d\omega)}^{p-\sigma}\le C\|K(\omega)\|_{L^p(d\omega)}^{p-\sigma}.
\end{align}
To estimate $I_2$, we use Corollary \ref{cor:seqdyad} to decompose $\bm{a}=(a_i)$ and $\bm{b}=(b_i)$ into dyadic level sets
\begin{align}\nonumber
\bm{a} = \sum_{m_3<2^l\le 10n_0}\lambda_{l}\bm{a}(l), &\qquad \bm{b} = \sum_{m_3<2^l\le 10n_0}\mu_{l}\bm{b}(l), \\
\label{eqn:lambdamu}\sum_l \lambda_l^2\le 1, &\qquad \sum_{l} \mu_{l}^2 \le 1,\\
\nonumber |\supp\bm{a}(l)|\le 2^{l}, &\qquad |\supp\bm{b}(l)|\le 2^{l}, \\
\nonumber |a_i(l)|\le 2^{-l/2}, &\qquad |b_i(l)|\le 2^{-l/2},
\end{align}
where $m_3 = |\supp \bm{c}|$.  Rewrite $f_{\bm{a}, \omega_1}$ and $f_{\bm{b}, \omega_2}$ in terms of level sets of $\bm{a}$ and $\bm{b}$ as above. For $d\ge 0$ an integer, we write $\mathcal L_{m,d} = \{(l,l^\prime): m\le2^l, 2^{l^\prime}\le 10n_0, |l-l^\prime|=d\}$ and estimate
\begin{align} 
\nonumber
| \langle f_{\bm a,\omega_1}, f_{\bm b,\omega_2}(1+|f_{\bm c,\omega_3}|)^{p-2} \rangle | &\le  \sum_{m_3< 2^l, 2^{l'} \le 10 n_0} \lambda_l \mu_{l'}  | \langle f_{\bm a(l),\omega_1}, f_{\bm b(l'),\omega_2}(1+|f_{\bm c,\omega_3}|)^{p-2} \rangle | 
\\  \nonumber
    &=\sum_{0\le d\le \log_2(\frac{10n_0}{m_3})} \sum_{(l, l') \in \mathcal{L}_{m_3,d}}  \lambda_l \mu_{l'}  | \langle f_{\bm a(l),\omega_1}, f_{\bm b(l'),\omega_2}(1+|f_{\bm c,\omega_3}|)^{p-2} \rangle |
\\  \nonumber
    &\le \sum_{0\le d\le \log_2(\frac{10n_0}{m_3})}\sum_{m_3<2^k\le 10n_0}\lambda_k\big(\mu_{k+d}+ \mu_{k-d})
\\   \nonumber
&\quad\times\sup_{(l, l') \in \mathcal{L}_{m_3,d}}  | \langle f_{\bm a(l),\omega_1}, f_{\bm b(l'),\omega_2}(1+|f_{\bm c,\omega_3}|)^{p-2} \rangle | \nonumber
\\ \nonumber
    &\lesssim \sum_{0\le d\le \log_2(\frac{10n_0}{m_3})}(\sum_k \lambda_k^2)^{1/2}(\sum_k\mu_k^2)^{1/2}
\\ \nonumber
&\quad\times\sup_{(l, l') \in \mathcal{L}_{m_3,d}}  | \langle f_{\bm a(l),\omega_1}, f_{\bm b(l'),\omega_2}(1+|f_{\bm c,\omega_3}|)^{p-2} \rangle | \nonumber
\\ \nonumber
    &\le \sum_{0\le d\le \log_2(\frac{10n_0}{m_3})} \sup_{(l, l') \in \mathcal{L}_{m_3,d}}  | \langle f_{\bm a(l),\omega_1}, f_{\bm b(l'),\omega_2}(1+|f_{\bm c,\omega_3}|)^{p-2} \rangle |
\\ \nonumber
    &\le \sum_{0\le d\le \log_2(\frac{10n_0}{m_3})} \sup_{ \substack{(l, l') \in \mathcal{L}_{m_3,d} \\ l=l'+d} }  | \langle f_{\bm a(l),\omega_1}, f_{\bm b(l'),\omega_2}(1+|f_{\bm c,\omega_3}|)^{p-2} \rangle |
\\ \label{oneparameter}
&\quad+ \sum_{0\le d\le \log_2(\frac{10n_0}{m_3})} \sup_{ \substack{(l, l') \in \mathcal{L}_{m_3,d} \\ l'=l+d} }  | \langle f_{\bm a(l),\omega_1}, f_{\bm b(l'),\omega_2}(1+|f_{\bm c,\omega_3}|)^{p-2} \rangle |,
\end{align}
where in the third--to--last inequality we used Cauchy--Schwarz and in the second--to--last inequality \eqref{eqn:lambdamu}.
Now, we can estimate
\begin{align*}
&\sup_{(\bm a,\bm b,\bm c)\in \mathcal{A}_2}|\langle f_{\bm a,\omega_1}, 	f_{\bm b,\omega_2}(1+|f_{\bm c,\omega_3}|)^{p-2}\rangle|
\\
&\quad=\sup_{1\le m_3 \le 10n_0} \sup_{\substack{\bm a, \bm b\in\Pi_{10n_0}, \bm c\in\Pi_{m_3} \\ \max_i( |a_i|, |b_i|)\le m_3^{-1/2}   }}  |\langle f_{\bm a,\omega_1}, 	f_{\bm b,\omega_2}(1+|f_{\bm c,\omega_3}|)^{p-2}\rangle|
\\
&\quad\le
\sup_{1\le m_3 \le 10n_0} \sum_{0\le d\le \log_2(\frac{10n_0}{m_3})}  \sup_{ \substack{(l, l') \in \mathcal{L}_{m_3,d} \\ l=l'+d} }   \sup_{\substack{\bm a \in \Pi_{2^l},\bm b \in \Pi_{2^{l'}}, \bm c\in\Pi_{m_3}   \\\max_i|a_i| \le 2^{-l/2}}}   |\langle f_{\bm a,\omega_1}, 	f_{\bm b,\omega_2}(1+|f_{\bm c,\omega_3}|)^{p-2}\rangle|
\\
&\qquad +\sup_{1\le m_3 \le 10n_0} \sum_{0\le d\le \log_2(\frac{10n_0}{m_3})}  \sup_{ \substack{(l, l') \in \mathcal{L}_{m_3,d} \\ l'=l+d} } 
 \sup_{\substack{\bm a \in \Pi_{2^l},\bm b \in \Pi_{2^{l'}}, \bm c\in\Pi_{m_3}   \\\max_i|b_i| \le 2^{-l'/2}}} 
 |\langle f_{\bm a,\omega_1}, 	f_{\bm b,\omega_2}(1+|f_{\bm c,\omega_3}|)^{p-2}\rangle|
\\
&\quad=:S_1(\omega_1,\omega_2,\omega_3)+S_2(\omega_1,\omega_2,\omega_3).
\end{align*}
To estimate $S_1(\omega_1,\omega_2,\omega_3)$ we note that for fixed $m_3\in[10n_0]$ and $d\ge 0$ we have
\begin{align*}
&\sup_{ \substack{(l, l') \in \mathcal{L}_{m_3,d} \\ l=l'+d} }    \sup_{\substack{\bm a \in \Pi_{2^l},\bm b \in \Pi_{2^{l'}}, \bm c\in\Pi_{m_3}   \\\max_i|a_i| \le 2^{-l/2}}}   |\langle f_{\bm a,\omega_1}, 	f_{\bm b,\omega_2}(1+|f_{\bm c,\omega_3}|)^{p-2}\rangle|
\\
&\quad\le \sup_{\substack{m_1, m_2\in [10n_0]\\ m_1 \ge 2^dm_2\ge 2^d m_3}}  \sup_{\substack{\bm a\in\Pi_{m_1} ,\bm b\in\Pi_{m_2}, \bm c\in\Pi_{m_3} \\ \max_i|a_i| \le m_1^{-1/2}   }}    |\langle f_{\bm a,\omega_1}, 	f_{\bm b,\omega_2}(1+|f_{\bm c,\omega_3}|)^{p-2}\rangle|
\\
&\quad\le \sup_{\substack{m_1, m_2\in [10n_0]\\m_1 \ge 2^d m_2\ge 2^d m_3}}\sup_{\substack{|A|\le m_1\\ \bm b\in \Pi_{m_2}, \bm c\in \Pi_{m_3}}} \frac{1}{\sqrt{m_1}} \sum_{i\in A}\xi_i(\omega_1) |\langle \varphi_i, f_{\bm b, \omega_2} (1+|f_{\bm c, \omega_3}|)^{p-2}\rangle|
\\
&\quad= \sup_{\substack{m_1, m_2\in [10n_0]\\m_1 \geq 2^dm_2\geq 2^d m_3}}K_{m_1,m_2,m_3}(\omega_1,\omega_2,\omega_3),
\end{align*}
see \eqref{eq:K_{m_1,m_2,m_3}}. Consequently, we get
\begin{align*}
S_1(\omega_1,\omega_2,\omega_3)\le \sup_{1\le m_3 \le 10n_0} \sum_{0\le d\le \log_2(\frac{10n_0}{m_3})}\sup_{\substack{m_1, m_2\in [10n_0]\\m_1 \geq 2^dm_2\geq 2^d m_3}}K_{m_1,m_2,m_3}(\omega_1,\omega_2,\omega_3).
\end{align*}
The analogous reasoning shows that 
\begin{align*}
S_2(\omega_1,\omega_2,\omega_3)&\le \sup_{1\le m_3 \le 10n_0} \sum_{0\le d\le \log_2(\frac{10n_0}{m_3})}\sup_{\substack{m_1, m_2\in [10n_0]\\m_2 \geq 2^dm_1\geq 2^d m_3}}K_{m_2,m_1,m_3}(\omega_2,\omega_1,\omega_3).
\end{align*}
Applying the trick from Proposition \ref{prop:sup-exchange} twice (first to the supremum over $m_3$ and later to the supremum over $(m_1,m_2)$), and then using \eqref{Kest}, we see that the above bound for $S_1$ implies that for fixed $\omega_2$ and $\omega_3$ we have
\begin{align*} 
\int S_1(\omega_1,\omega_2,\omega_3) \,d\omega_1 &\lesssim\sup_{1\le m_3 \le 10n_0}  \sum_{0\le d\le \log_2(\frac{10n_0}{m_3})} \sup_{\substack{m_1, m_2\in [10n_0]\\m_1 \geq 2^dm_2\geq 2^d m_3}}  \|   K_{m_1, m_2,  m_3}(\omega_1,\omega_2,\omega_3) \|_{L^{q_0}(d\omega_1)}
\\
&\lesssim 
\sup_{1\le m_3 \le 10n_0}  \sum_{0\le d\le \log_2(\frac{10n_0}{m_3})} \sup_{\substack{m_1, m_2\in [10n_0]\\m_1 \geq 2^dm_2\geq 2^d m_3}}  (\delta m_3^{p/2-1} + \frac{m_2+m_3}{m_1})^{1/2}
\\
&\qquad\qquad\qquad\qquad\qquad\qquad\qquad\times(1+K(\omega_2) +K(\omega_3))^{p-\sigma}
\\
&\lesssim \sup_{1\le m_3 \le 10n_0} \sum_{0\le d\le \log_2(\frac{10n_0}{m_3})} ((\delta m_3^{p/2-1})^{1/2}+2^{-d/2})
\\
&\qquad\qquad\qquad\qquad\qquad\qquad\qquad\times(1+K(\omega_2) +K(\omega_3))^{p-\sigma}
\\
&\lesssim  (1+K(\omega_2) +K(\omega_3))^{p-\sigma},
\end{align*}
where we used $\delta=n_0/n$, $n_0=n^{2/p}$ to get the final inequality. A similar argument shows that for fixed $\omega_1$ and $\omega_3$, one has
\begin{align*} 
\int S_2(\omega_1,\omega_2,\omega_3) \,d\omega_2 \lesssim (1+K(\omega_1) +K(\omega_3))^{p-\sigma}.
\end{align*}
Now, we are ready to estimate $I_2$---see \eqref{eq:I12def}. Using the above bounds and H\"older's inequality, we get
\begin{align*}
I_2&=\iiint  \sup_{(\bm a,\bm b,\bm c)\in \mathcal{A}_2}|\langle f_{\bm a,\omega_1}, 	f_{\bm b,\omega_2}(1+|f_{\bm c,\omega_3}|)^{p-2}\rangle|\,d\omega_1\,d\omega_2\,d\omega_3
\\
&\le 
\iiint \big( S_1(\omega_1,\omega_2,\omega_3)+S_2(\omega_1,\omega_2,\omega_3) \big) \,d\omega_1\,d\omega_2\,d\omega_3
\\
&\lesssim 
\iint (1+K(\omega_2) +K(\omega_3))^{p-\sigma} d\omega_2\,d\omega_3+\iint (1+K(\omega_1) +K(\omega_3))^{p-\sigma} \,d\omega_1 \,d\omega_3
\\
&\lesssim  \|K(\omega)\|_{L^p(d\omega)}^{p-\sigma}.
\end{align*}

Combining the above estimates for $I_1$ and $I_2$ with Proposition \ref{prop:dec}, we get
$$\|K(\omega)\|_{L^p(d\omega)}^{p}\le C(\|K(\omega)\|_{L^p(d\omega)}^{p-1} + \|K(\omega)\|_{L^p(d\omega)}^{p-\sigma}),$$
so $\|K(\omega)\|_{L^p(d\omega)}\lesssim 1$, which is \eqref{eqn:random-lambdap}.

\subsection{The case $p\ge4$.}\label{subsec:p>4} The argument in the case $2<p<4$ cannot be applied directly to the case $p\geq 4$ as the proof of \eqref{Kest} as presented in Section \ref{sec:end}, does not work in this case. Bourgain was able to overcome this issue by using an inductive argument. 
 We assume that Theorem \ref{thm:main} holds for some exponent $p_1$ and prove that it follows for $p$ satisfying $p/2< p_1 < p$. The base for the induction will be Theorem \ref{thm:main2} with $2<p<4$ (this case will be proved independently).

We need to prove that a random subset of $[n]$ of size $\simeq n^{2/p}$ has the $\Lambda(p)$-property. 
To this end consider $\delta'>0$ such that $n^{2/p}=\delta'n^{2/{p_1}}$ and let $S_1$ be a random subset of $[n]$ of size $n_1=|S_1|\simeq n^{2/{p_1}}$.  By the induction hypothesis, with high probability $S_1$ satisfies the $\Lambda(p_1)$-property
\begin{equation}\label{eqn:lambda-p1}
   \|\sum_{i\in S_1} a_i\varphi_i\|_{p_1}\le C|\mathbf{a}|, \qquad \mathbf{a}\in\R^n.
\end{equation}
Now let $S\subset S_1$ be a random subset of $S_1$ constructed using a collection of selectors of mean $\delta'$. Then the expected size of $S$ is $|S|\simeq n^{2/p}=\delta'n_1$ and consequently, the proof of Theorem \ref{thm:main2} reduces to showing that $S$ has the $\Lambda(p)$-property. Following the same overall strategy as in the case $2< p < 4$, but with the supports of sequences restricted to $S_1$, one can in fact show that the generic random subset $S_\omega\subset S_1$ of size $\simeq n^{2/p}$ has the $\Lambda(p)$-property.

To make this restriction on the sequence supports to $S_1$ explicit, we define:
\[
K_{S_1} = \sup_{|\mathbf a|\le 1}\|\sum_{i\in S_1}a_i\varphi_i\|_{p}.
\]
Now let $S_\omega\subset S_1$ be a random subset of $S_1$ obtained using independent selectors of mean $\delta'$. We define $K_{S_1}(\omega)$ by restricting the summation in the definition of $K_{S_1}$ to $S_\omega$, specifically:
\[
K_{S_1}(\omega) = \sup_{|\mathbf a|\le 1}\|\sum_{i\in S_\omega}a_i\varphi_i\|_{p}.
\]

Assuming that \eqref{eqn:lambda-p1} holds for $S_1\subset[n]$ of size $n^{2/p_1}$ with  $p_1$ satisfying $p/2< p_1 < p$, we will establish in Section \ref{sec:end} the following analog of the key inequality \eqref{Kest}
\begin{equation}\label{eq:Kest-p>4}
\|K^{S_1}_{m_1,m_2,m_3}(\omega_1,\omega_2,\omega_3)\|_{L^{q_0}(d\omega_1)}\lesssim (\delta' m_3^{p/p_1-1}+\frac{m_2+m_3}{m_1})^{1/2}(1+K_{S_1}(\omega_2)+K_{S_1}(\omega_3))^{p-\sigma},
\end{equation}
with $\delta'=n^{2/p-2/p_1}$, where 
$$
K^{S_1}_{m_1,m_2,m_3}(\omega_1,\omega_2,\omega_3):=\sup_{\substack{A\subset S_1 \\ |A|\le m_1}}\sup_{\bm b\in \Pi^{S_1}_{m_2}}\sup_{\bm c\in \Pi^{S_1}_{m_3}}\frac{1}{\sqrt{m_1}}\sum_{i\in A} \xi_i(\omega_1)|\langle \varphi_i, 	f_{\bm b,\omega_2}|f_{\bm c,\omega_3}|^{p-2}\rangle |
$$
and
\begin{align}\label{eq:PimS1}
&\Pi^{S_1}_{m}=\{\mathbf{a}=(a_i)_{i\in S_1}:|\mathbf{a}|\le 1, |\supp \mathbf{a}|\le m\}.		
\end{align}
Then, by following the same approach as in the case $p\in(2,4)$, but with definitions adjusted for vectors supported on the fixed set $S_1$, one can check that \eqref{eq:Kest-p>4} implies  $\|K_{S_1}(\omega)\|_{L^p(d\omega)}\lesssim 1$. We remark that to obtain the estimate corresponding to \eqref{eqn:Kestcompute} one needs to use the fact that $\delta'n_0^{p/p_1-1}=O(1)$, with $n_0=n^{2/p}$.

\section{A probabilistic inequality}\label{sec:prob}

In this section we will establish an important probabilistic inequality which will be a key tool in the proof of \eqref{Kest} and \eqref{eq:Kest-p>4}. The inequality controls the $L^{q_0}$-norm of an \emph{uncountable} supremum of random variables indexed by points in the positive orthant of $\R^n$ in terms of the entropy of its index set.

For $\mathbf{x}=(x_1,\dots,x_n)\in \R^n$, denote $|\mathbf{x}|=(\sum_{i=1}^nx_i^2)^{1/2}$. If $\mathcal{E}\subset \R^n$ and $t>0$, denote by $N_2(\mathcal{E},t)$ the metric entropy number with respect to the $\ell^2$-distance, i.e., the minimum number of $\ell^2$-balls of radius $t$ needed to cover $\mathcal{E}$.

\begin{lem}\label{lem:1lem1}
    Let $\mathcal{E}$ be a subset of $\R^n_+$ and $B=\sup_{\mathbf{x}\in \mathcal{E}}|\mathbf{x}|$. Let $0<\delta<1$ and $(\xi_i)_{i=1}^n$ be a family of independent $\{0, 1\}$--valued random variables (= selectors) of mean $\delta=\int \xi_i(\omega)\,d\omega$. Then for $m\in[n]$, and $1\le q_0 < \infty$, the estimate
    \begin{align}\label{eqn:1mainlemma}
      \left\|  \sup_{\mathbf{x}\in \mathcal{E}, |A|\le m}\Big[\sum_{i\in A} \xi_i(\omega)x_i\Big] \right\|_{L^{q_0}(d\omega)}  &\lesssim\Big[\delta m + \frac{q_0}{\log\frac{1}{\delta}}\Big]^{1/2}B
      \\ \nonumber
      &\quad+(\log\frac{1}{\delta})^{-1/2} \int_0^B\Big[\log N_2(\mathcal{E},t)\Big]^{1/2}\,dt,
    \end{align}
holds with an implicit absolute constant independent of $n,m,q_0$.
\end{lem}
We remark that we will apply estimate \eqref{eqn:1mainlemma} with $q_0=\log n$, so it is important that the constant $C$ does not depend on $q_0$.

\begin{rem}

To understand the significance of the factor $\big[\log N_2(\mathcal{E},t)\big]^{1/2}$ in the expression arising in the estimate of Lemma \ref{lem:1lem1} it is useful to consider a simplified model. Assume that $(X_i)_{i\in [N]}$ is a family of random variables such that $\mathbb{E}X_i=1$ for each $i\in[N]$ and satisfying the following concentration of measure inequality
$$
\mathbb P(|X_i-1|>K)\le e^{-K^2}, \qquad K>0.
$$

Suppose that we are interested in estimating the expected value of the supremum of $\{X_n: n\in[N]\}$. When $N\gg1$, a naive bound $\mathbb{E}[\sup_{n\in[N]} X_n]\le N$ is overly crude because if it were accurate, it implies that at least one variable deviates from the mean by roughly $N$. However, the probability that any of the $X_n$ have such a large deviation (via the union bound) is at most $Ne^{-N^2}$, which is negligible. Hence, this estimate almost surely overstates the supremum.
Instead, the supremum is governed by the rare events where $X_n$ deviates significantly above its mean. If K is too large, deviations by K are exceedingly rare, so the supremum cannot realistically scale with K. Conversely, if K is too small, such deviations are common, and we underestimate the supremum.
The balance occurs at $K=\sqrt{\log N}$ , where the expected number of deviations by $K$ is approximately $1$, that is $Ne^{-K^2}\simeq 1$. This is the point where rare large deviations are likely enough (and large enough) to significantly influence the supremum. 
Thus, the term \( \sqrt{\log N} \) represents a balance between \textit{complexity} (the number or entropy of the collection of random variables) and the \textit{concentration} of the individual random variables. As \( N \) increases, the supremum is likely to exhibit larger deviations, but the growth is controlled by \( \sqrt{\log N} \). 

For a comprehensive discussion of the duality of complexity and concentration, as well as the techniques used to estimate uncountable suprema of random variables, we refer the reader to Talagrand's book \cite{T}.
\end{rem}

Before we present the proof of the above lemma we will need a simpler probabilistic inequality which essentially captures the large deviations of the individual random variables inside the supremum in Lemma \ref{lem:1lem1}. Its proof will require the following Bernstein's inequality, see \cite[Theorem 3]{BLB}.
\begin{lem}[Bernstein]
Let $X_1,\dots,X_l$ be independent mean-zero random variables. Suppose that $|X_i|\le 1$ almost surely, for all $i\in[l]$. Then, for all $u>0$,
\begin{equation}
\Prob\Big(\sum_{i=1}^l X_i\ge u\Big)\le \exp\Big(-\frac{\frac12 u^2}{\sum_{i=1}^l\E[X_i^2] + \frac13 u}\Big). \label{eq:bernstein}
\end{equation}
\end{lem}

\begin{lem}\label{lem:1lem2}
 Let $(\xi_i)$ be as in Lemma \ref{lem:1lem1}. Then for $q\geq1$
    \begin{align}
        \big\|  \sum_{i=1}^l\xi_i(\omega) \big\|_{L^q(d\omega)}\lesssim \delta l +\dfrac{q}{\log (2+q/\delta l)},
    \end{align}
with the implicit constant independent of $l,\delta$ and $q$.
\end{lem}

\begin{proof}
We begin with showing that if  $q\le 2\delta l$, then 
    \begin{align*}
        \big\|  \sum_{i=1}^l\xi_i(\omega) \big\|_{L^q(d\omega)}\lesssim\delta l.
    \end{align*}
Note that in view of H\"older's inequality, the above relationship in fact means that for $q\le 2\delta l$ one has
\begin{align*}
        \big\|  \sum_{i=1}^l\xi_i(\omega) \big\|_{L^q(d\omega)}\simeq\big\|  \sum_{i=1}^l\xi_i(\omega) \big\|_{L^1(d\omega)}.
\end{align*}

Letting $\E f = \int_{} f(\omega)\, d\omega$ we write
\begin{align*}
\big\|  \sum_{i=1}^l\xi_i(\omega) \big\|_{L^q(d\omega)}^q &= \E\Big(\sum_{i=1}^l\xi_i\Big)^q = q\int_0^\infty \lambda^{q-1}\Prob\Big(\sum_{i=1}^l\xi_i > \lambda\Big)\,d\lambda
\\
&\le q\int_0^{5\delta l}\lambda^{q-1}\,d\lambda + q\int_{5\delta l}^\infty \lambda^{q}\,\Prob\Big(\sum_{i=1}^l\xi_i>\lambda\Big)\,d\lambda \\
&= (5\delta l)^q + q\int_{5\delta l}^\infty \lambda^{q}\,\Prob\Big(\sum_{i=1}^l\xi_i>\lambda\Big)\,d\lambda.
\end{align*}
Let 
$$
I := q\int_{5\delta l}^\infty \lambda^{q}\,\Prob\Big(\sum_{i=1}^l\xi_i>\lambda\Big)\,d\lambda
$$ 
be the second term in the last expression above.  To estimate it we first rewrite
\[
\Prob\Big(\sum_{i=1}^l\xi_i>\lambda\Big) = \Prob\Big(\sum_{i=1}^l X_i > \lambda-\delta l\Big),
\]
where each $X_i := \xi_i - \delta$ is a mean-zero random variable. By changing the variable $\lambda = \delta l + u$, we get
\[
I = q\int_{4\delta l}^\infty (\delta l + u)^{q}\,\Prob\Big(\sum_{i=1}^lX_i > u\Big)\,du.
\]
Since $\delta l \le \frac14u$ on the region of integration, we can estimate
\begin{equation}\label{eq:bound_I} 
I \le q\cdot\left(\frac54\right)^{q}\int_{4\delta l}^\infty u^{q}\,\Prob\Big(\sum_{i=1}^l X_i > u\Big)\,du.
\end{equation}
Now, we apply Bernstein's inequality \eqref{eq:bernstein} to estimate the tail probability.
We note that the variables $X_i$ are $\{-\delta,1-\delta\}$-valued, so $|X_i|\le 1$ almost surely, for all $i$, and by a quick calculation,
\[
\E[X_i^2] = \E[(\xi_i - \delta)^2] = \delta - \delta^2 \le \delta.
\]
Plugging $\E[X_i^2]\le \delta$ into the right-hand side of Bernstein's inequality and noting that on the region of integration again, $\delta l \le \frac14u < \frac13u$, we get an upper bound of
\begin{align*}
\Prob\Big(\sum_{i=1}^l X_i > u\Big)\le\exp\Big(-\frac{\frac12 u^2}{\delta l + \frac13 u}\Big)\le\exp\Big(-\frac{\frac12 u^2}{\frac13 u + \frac13 u}\Big) = e^{-\frac34 u}.
\end{align*}
Now we plug this estimate of the tail probability back into \eqref{eq:bound_I} to get
\begin{equation} \label{eqn:bound_I} 
I \le q\cdot \left(\frac54\right)^{q}\int_{4\delta l}^\infty u^{q}e^{-\frac34 u}\,du.
\end{equation}
Next notice that for $u > 0$ one has $u^q\le (2q)^qe^{\frac12u}$.
To see that, it suffices to verify that  $f(u) = u^qe^{-\frac12u}$ has a global maximum at $u = 2q$.
Plugging this inequality into \eqref{eqn:bound_I} gives
\[
I\le q\cdot \left(\frac5{2}\right)^{q}q^q\int_{4\delta l}^\infty e^{-\frac14u}\,du = q\cdot \left(\frac5{2}\right)^{q}q^q\big[4e^{-\delta l}\big] \le 4q\cdot \left(\frac52\right)^qq^q.
\]
Now we use the assumption $q\le 2\delta l$ to write
\[
I\le 4q\cdot \left(\frac5{2}\right)^q(2\delta l)^q = 4q\cdot 5^q(\delta l)^q.
\]
Altogether, we have shown
\[
\big\|\sum_{i=1}^l\xi_i(\omega)\big\|_{L^q(d\omega)}^q \le (5^q + 4q\cdot 5^q)(\delta l)^q = (1+4q)\cdot (5\delta l)^q.
\]
Taking $q^{th}$ roots and using $(1+4q)^{\frac1q}\le 5$ for $q\ge 1$, we see that
\[
\big\|\sum_{i=1}^l\xi_i(\omega)\big\|_{L^q(d\omega)} \le 5^2\,\delta l,
\]
as desired.
It finishes the proof in the case $q\le 2\delta l$.

It remains to treat the case $q>2\delta l$. Without loss of generality we can assume that $q$ is an integer. It suffice to prove that the estimate
    \begin{align}\label{eq:qlarge}
    \big\|\sum_{i=1}^l\xi_i(\omega)\big\|_{L^q(d\omega)}^q\le C^q \Big(\frac{q}{\log (2+\frac{q}{\delta l})}\Big)^q,
    \end{align}
holds with some constant $C>0$ independent of $q,l$ and $\delta$. We remark that in what follows we will let $C$ vary from line to line. In particular, $C^q$ will absorb all multiplicative factors of  polynomial growth $q^K$ for any $K>0$. 

Note that $\sum_{i=1}^l \xi_i$ can be viewed as the random variable counting the number of successes in a series of Bernoulli trials of length $l$ with the probability of success $\delta$. Therefore, we have
\begin{align*}
   \big\|\sum_{i=1}^l\xi_i(\omega)\big\|_{L^q(d\omega)}^q&=\int \Big[\sum_{i=1}^l \xi_i(\omega)\Big]^q d\omega=\sum_{k=1}^l \binom{l}{k}\delta^k (1-\delta)^{l-k}k^q. 
\end{align*}
 Using the fact that $\binom{l}{k}\le (\frac{l}{k})^ke^k$ and bounding $1-\delta$ trivially by $1$ we get
\begin{align*}
\big\|\sum_{i=1}^l\xi_i(\omega)\big\|_{L^q(d\omega)}^q\le C^q\sum_{k=1}^l\Big(\frac{\delta l}{k}\Big)^k k^q =  C^q\sum_{k=1}^l F(k),
\end{align*}
where $F(x) = (\frac{\delta l}{x})^x x^q$. In view of Lemma \ref{lem:F} from the Appendix (applied with $\kappa=\delta l < q$), the function $F$ changes its monotonicity only once on $(1,\infty)$. Thus, we can dominate the sum in $k$ by the integral $\int_{0}^\infty F(x)\,dx$ (by two copies of it, strictly speaking). Further, changing the variable $q\alpha=x$ and absorbing the Jacobian $q$ into $C^q$ we get
\begin{align*}
\big\|\sum_{i=1}^l\xi_i(\omega)\big\|_{L^q(d\omega)}^q&\le C^q\int_{0}^\infty F(x)\,dx=C^q\int_{0}^\infty \Big(\frac{\delta l}{x}\Big)^x x^q\,dx=  C^q  q^q\int_0^\infty \Big(\frac{\delta l}{\alpha q}\Big)^{\alpha q} \alpha^q\, d\alpha\\
&\le C^q q^q\int_0^\infty \Big(\frac{\delta l}{ q}\Big)^{\alpha q} \alpha^q \,d\alpha =  C^q q^q\int_0^\infty e^{-\alpha q\log \frac{ q}{\delta l}} \alpha^q \,d\alpha
\end{align*}
where the second inequality uses the bound $(\frac{1}{\alpha})^\alpha \le C$, valid for any $\alpha>0$. By changing the variable $u=\alpha q \log\frac{q}{\delta l}$, and absorbing the Jacobian factor into $C^q$ again, we obtain
\begin{align*}
\big\|\sum_{i=1}^l \xi_i\big\|_{L^q(d\omega)}^q 
&= C^q q^q \Big(\frac{1}{q \log \frac{q}{\delta l}}\Big)^{q} \int_0^\infty e^{-u} u^q \,du = C^q q^q \Big(\frac{1}{q \log \frac{q}{\delta l}}\Big)^{q} \Gamma(q-1).
\end{align*}
In view of the known asymptotics for the Gamma function
$$
\Gamma(n)=n^{n-1/2}e^{-n}\sqrt{2\pi}\big(1+O(1/n)\big), \qquad n\rightarrow \infty,
$$
we can finally estimate
\begin{align*}
\big\|\sum_{i=1}^l\xi_i(\omega)\big\|_{L^q(d\omega)}^q&\le C^q q^q \Big(\frac{1}{q \log \frac{q}{\delta l}}\Big)^{q} q^{-1/2} \frac{q^q}{e^q}
\le C^q \Big(\frac{q}{\log (\frac{q}{\delta l})}\Big)^q,
\end{align*}
which gives \eqref{eq:qlarge}, and concludes the proof of the lemma.
\end{proof}

\begin{definition}
    A \emph{$\delta$-net} of a metric space $(X,d)$ is any subset $X'\subset X$ such that for each $x\in X$, there is $x'\in X'$ so that $d(x,x')<\delta$. A nonempty $\delta$-net is \emph{minimal} if any proper subset is not a $\delta$-net.
\end{definition}
We will need the following result connecting $\delta$-nets to the metric entropy.
\begin{prop}\label{prop:netentropy}
    If $(X,d)$ is a bounded metric space, the cardinality of a minimal $\delta$-net is equal to the metric entropy number $N_d(X,\delta)$. 
\end{prop}

\begin{proof}[Proof of Lemma \ref{lem:1lem1}]
Without a loss of generality, suppose that $B = 2^{-k_0}$ for some $k_0\in\mathbb Z$. For each $k\ge k_0$, let $\mathcal{E}_k$ be a minimal $2^{-k}$-net for $\mathcal{E}$. By the definition of a 
net, given $\mathbf{x}\in \mathcal E$, and $k\ge k_0$, we can find a point $\mathbf{x}^{(k)}\in\mathcal E_k$ such that $|\mathbf{x}-\mathbf{x}^{(k)}|<2^{-k}$, so the points $\mathbf{x}^{(k)}$ then converge rapidly to $\mathbf{x}$ as $k\to\infty$. Setting $\mathbf{x}^{(k_0-1)} = 0$ and using a telescoping sum, we can write
\begin{equation*}
    \mathbf{x} = \lim_{k\to\infty}\mathbf{x}^{(k)} = \sum_{k=k_0}^\infty (\mathbf{x}^{(k)}-\mathbf{x}^{(k-1)}).
\end{equation*}
We normalize things to scale $\simeq 1$ by defining, for each $k$, a new collection
\[
\mathcal F_k = \{2^k(\mathbf{x}^{(k)} - \mathbf{x}^{(k-1)}):(\mathbf{x}^{(k)},\mathbf{x}^{(k-1)})\in\mathcal E_k\times\mathcal E_{k-1},2^k|\mathbf{x}^{(k)}-\mathbf{x}^{(k-1)}|< 4\},
\]
so each $\mathbf{x}\in\mathcal E$ has a representation as a sum
\[
\mathbf{x} = \sum_{k=k_0}^\infty 2^{-k}\mathbf{y}^{(k)},
\]
where $\mathbf{y}^{(k)}$ are taken from $\mathcal F_k$. To see that $2^k(\mathbf{x}^{(k)}-\mathbf{x}^{(k-1)})$ in the representation of an $\mathbf{x}\in\mathcal E$ indeed belongs to $\mathcal F_k$, we use the triangle inequality
\[
2^k|\mathbf{x}^{(k)}-\mathbf{x}^{(k-1)}|\le 2^k(|\mathbf{x}^{(k)}-\mathbf{x}| + |\mathbf{x}^{(k-1)}-\mathbf{x}|)\le2^k( 2^{-k}+2^{-k+1})< 4,
\]
so that $2^k(\mathbf{x}^{(k)}-\mathbf{x}^{(k-1)})\in\mathcal F_k\subset B(0,4)$. Note that to justify the above estimate in the boundary case $k=k_0$ one needs to use the assumption $\mathbf{x}^{(k_0-1)} = 0$ together with the condition $|\mathbf{x}|\le B =2^{-k_0}$.

As the nets $\mathcal E_k$ are minimal, we can use Proposition \ref{prop:netentropy} to get the estimate
\[
|\mathcal F_k|\le |\mathcal E_k|\cdot|\mathcal E_{k-1}| \le |\mathcal E_k|^2 = N_2(\mathcal E,2^{-k})^2
\]
which implies
\begin{align}\label{eqn:entropydef}
    \log |\mathcal F_k| \le 2 \log N_2(\mathcal{E},2^{-k}).
\end{align}
By taking the entropy information of $\mathcal E$ into account this way, we can replace the uncountable supremum on the left-hand side of the equation \eqref{eqn:1mainlemma} by a more manageable expression involving only suprema over finite sets. To do this, we write, for any fixed $\mathbf{x} = (x_1,\dots,x_n)\in\mathcal E$ and $|A|\le m$,
\begin{align*}
    \sum_{i\in A}\xi_ix_i 
&= \sum_{i\in A}\xi_i\sum_{k= k_0}^\infty 2^{-k}y_i^{(k)}
    \le \sum_{i\in A}\xi_i\sum_{k=k_0}^\infty 2^{-k}|y_i^{(k)}|
    = \sum_{k=k_0}^\infty 2^{-k}\sum_{i\in A}\xi_i|y_i^{(k)}| 
\\
&\le \sum_{k=k_0}^\infty 2^{-k}\sup_{y\in\mathcal F_k,|A|\le m}\sum_{i\in A}\xi_i|y_i|,
\end{align*}
with the  inequality in the second line holding because the random variables $\xi_i$ are nonnegative. Thus,
\begin{align*}
     \left\|  \sup_{\mathbf{x}\in \mathcal{E}, |A|\le m} \Big(\sum_{i\in A} \xi_i(\omega)x_i\Big)  \right\|_{L^{q_0}(d\omega)} \le \sum_{k=k_0}^\infty 2^{-k}  \left\|  \sup_{\mathbf{y}\in \mathcal{F}_k, |A|\le m}\Big(\sum_{i\in A} \xi_i(\omega)|y_i|\Big) \right\|_{L^{q_0}(d\omega)}.
\end{align*}
It now suffices to show for any finite set $\mathcal F\subset B(0,4)\cap \R^n_+$,
\begin{align}\label{eq:supF}
    \left \|  \sup_{\mathbf{y}\in \mathcal{F}, |A|\le m}\Big(\sum_{i\in A} \xi_i(\omega)y_i\Big) \right\|_{L^{q_0}(d\omega)} \le C \sqrt{\delta m} +C\left(\log \frac{1}{\delta}\right)^{-1/2} \Big[q_0+\log |\mathcal{F}|\Big]^{1/2}.
\end{align}
Indeed, assuming momentarily that \eqref{eq:supF} holds, we can apply \eqref{eqn:entropydef}, sum over $k\ge k_0$, and use the inequality $\sqrt{a}+\sqrt{b}\le 2^{1/2}\sqrt{a+b}$, valid for $a,b>0$ to get 
\begin{align*}
     \left\|  \sup_{\mathbf{y}\in \mathcal{F}, |A|\le m}\Big(\sum_{i\in A} \xi_i(\omega)y_i\Big) \right\|_{L^{q_0}(d\omega)} &\le C\Big[\delta m 
 + \frac{q_0}{\log\frac{1}{\delta}}\Big]^{1/2}B
\\
&\quad+C\left(\log\frac{1}{\delta}\right)^{-1/2} \sum_{k=k_0}^\infty2^{-k}\Big[\log N_2(\mathcal E,2^{-k})\Big]^{1/2}.
\end{align*}

To obtain \eqref{eqn:1mainlemma} we need to replace a sum in $k$ by an integral. To this end notice that the function $t\mapsto [\log N_2(\mathcal E,t)]^{1/2}$ is nonincreasing,  so for any $s\in (k,k+1)$ we have 
$$
2^{-k}[\log N_2(\mathcal E,2^{-k})]^{1/2}\le 2^{-s+1}[\log N_2(\mathcal E,2^{-s})]^{1/2}.
$$
Consequently, 
\begin{align*}
    \sum_{k=k_0}^\infty2^{-k}\Big[\log N_2(\mathcal E,2^{-k})\Big]^{1/2} &\le 2\sum_{k=k_0}^\infty \int_{k}^{k+1}2^{-s}\Big[\log N_2(\mathcal E,2^{-s})\Big]^{1/2}\,ds
\\  
&=2 \int_{k_0}^{\infty}2^{-s}\Big[\log N_2(\mathcal E,2^{-s})\Big]^{1/2}\,ds.
\end{align*}
We change variables $t=2^{-s}$ to transform the last integration into 
\begin{align*}
     \int_{k_0}^{\infty}2^{-s}\Big[\log N_2(\mathcal E,2^{-s})\Big]^{1/2}\,ds =\frac1{\log 2} \int_0^B \Big[\log N_2(\mathcal E,t)\Big]^{1/2}\,dt
\end{align*}
which is the desired form.

It remains to show \eqref{eq:supF}. Let us define $\rho_1=\delta^{1/2}m^{-1/2}$ and $\rho_2=(\log \frac{1}{\delta})^{1/2}q^{-1/2}$ with $q=q_0+\log|\mathcal{F}|$. For $|A|\le m$, we can write
\begin{align}\nonumber
    \sum_{i\in A}\xi_i y_i &\le \sum_{y_i\geq \rho_2}y_i + \sum_{i\in A, y_i \le \rho_1}y_i +\sum_{\rho_1< y_i <\rho_2}\xi_iy_i \\\label{eqn:1csci}
    &\lesssim \sum_{y_i\geq \rho_2}y_i +m\rho_1+\sum_{\rho_1<y_i<\rho_2}\xi_iy_i.
\end{align}

Noting that $|\mathbf{y}|\lesssim 1$, we use Cauchy--Schwarz and Chebyshev's inequalities to estimate
\begin{align*}
    \sum_{y_i\geq \rho_2}y_i \le |\mathbf y| \cdot|\{i: y_i\geq \rho_2\}|^{1/2} \lesssim \rho_2^{-1}.
\end{align*}

Combining this with \eqref{eqn:1csci}, we have
\begin{align*}
    \sum_{i\in A}\xi_i y_i \lesssim \rho_2^{-1} +m\rho_1+\sum_{\rho_1<y_i<\rho_2}\xi_iy_i.
\end{align*}

Following the argument by Mockenhaupt and Schlag, see \cite{MS}, we write that 
\begin{align}\nonumber
     \left\|  \sup_{\mathbf{y}\in \mathcal{F}, |A|\le m}\Big(\sum_{i\in A} \xi_i(\omega)y_i\Big) \right\|_{L^{q_0}(d\omega)} 
     &\lesssim \rho_2^{-1} +m\rho_1+  \left\| \sup_{\mathbf{y}\in\mathcal{F}}\sum_{\rho_1<y_i<\rho_2}\xi_i(\omega)y_i \right\|_{L^{q_0}(d\omega)} \\
     \nonumber  
&\le  \rho_2^{-1} +m\rho_1+  \left\| \Big(\sum_{\mathbf{y}\in\mathcal{F}}\Big[\sum_{\rho_1<y_i<\rho_2}\xi_i(\omega)y_i\Big]^q\Big)^{1/q} \right\|_{L^{q_0}(d\omega)} \\
    \nonumber
     &\le  \rho_2^{-1} +m\rho_1+  \Big(\sum_{\mathbf{y}\in\mathcal{F}} \Big\|\sum_{\rho_1<y_i<\rho_2}\xi_i(\omega)y_i\Big\|^q_{L^q(d\omega)}\Big)^{1/q}  \\
    \nonumber   
     &\le  \rho_2^{-1} +m\rho_1+ |\mathcal{F}|^\frac1q\cdot \sup_{\mathbf{y}\in\mathcal{F}} \left\|\sum_{\rho_1<y_i<\rho_2}\xi_i(\omega)y_i\right\|_{L^{q}(d\omega)} \\
     \label{eqn:1supchange4}      
&\lesssim  \rho_2^{-1} +m\rho_1+ \sup_{|\mathbf{y}|\le 4} \left\|\sum_{\rho_1<y_i<\rho_2}\xi_i(\omega)y_i\right\|_{L^{q}(d\omega)}.
\end{align}
We used the embedding of $\ell^q(\mathcal{F})$ into $\ell^\infty(\mathcal{F})$ to show the first inequality above; in the second and the third we used H\"older's $(q>q_0)$ and triangle inequalities, and finally in the last inequality we used the condition $q=q_0+\log|\mathcal{F}|$ so that
\begin{align*}
|\mathcal{F}|^{\frac1q}=\exp(q^{-1}\log|\mathcal{F}|)\le e.
\end{align*}
Note the technique of exchanging the supremum over $\mathcal F$ from inside the $L^{q_0}(d\omega)$-norm to the outside of the the $L^{q_0+\log|\mathcal F|}(d\omega)$-norm. 

We now aim at showing 
\begin{equation}\label{eq:intermed}
 \sup_{|\mathbf{y}|\le 4}\Big\|\sum_{\rho_1<y_i<\rho_2}\xi_i(\omega)y_i\Big\|_{L^q(d\omega)}\lesssim
\delta \rho_1^{-1}+ q\rho_2(\log 1/ \delta)^{-1}.
\end{equation}
To this end notice that since $|\mathbf{y}|\lesssim 1$, the level set $\{i: y_i \simeq \frac{1}{\sqrt{l}}\}$ has  cardinality $|\{i:y_i \simeq \frac{1}{\sqrt{l}}\}|\lesssim l$. Using this fact and Lemma \ref{lem:1lem2}, we estimate 
\begin{align*}\nonumber
  \sup_{|\mathbf{y}|\le 4}\Big\|\sum_{\rho_1<y_i<\rho_2}\xi_i(\omega)y_i\Big\|_{L^q(d\omega)} &\lesssim  \sum_{\substack{l\ \text{dyadic} \\ \rho_2^{-2}<l<\rho_1^{-2}}} \frac{1}{\sqrt{l}}\Big\|\sum_{i=1}^l \xi_i(\omega)\Big\|_{L^q(d\omega)} \\
    &\lesssim \delta\rho_1^{-1}+q\sum_{\substack{ l\ \text{dyadic}\\ l>\rho_2^{-2}}} \frac{1}{\sqrt{l}\log (2+\frac{q}{\delta l})},
\end{align*}
so to prove \eqref{eq:intermed} it suffices to show
\begin{equation}\label{eq:ldyadic}
\sum_{\substack{ l\ \text{dyadic}\\ l>\rho_2^{-2}}} \frac{1}{\sqrt{l}\log (2+\frac{q}{\delta l})}\lesssim \rho_2(\log 1/ \delta)^{-1}.
\end{equation}
We change the variable $k:=\log_2 l$ in the sum on the left-hand side and then, arguing in a similar way as earlier in the proof, we can estimate the sum by the integral as follows
\begin{align*}
    \sum_{\substack{ l\ \text{dyadic}\\ l>\rho_2^{-2}}}\frac{1}{\sqrt{l}\log (2+\frac{q}{\delta l})} 
    &=\sum_{ k>\log_2 \rho_2^{-2}}\frac{1}{(2^k)^{1/2}\log (2+\frac{q}{\delta2^k})}
    &\lesssim \int_{\log_2 \rho_2^{-2}}^\infty \frac{1}{2^{s/2}\log (2+\frac{q}{\delta 2^s})} \,ds.
\end{align*}
Next, we change the variable of integration $t:= \rho_2^2 2^s$ and use the definition of  $\rho_2$ to further estimate
\begin{align*}
\int_{\log_2 \rho_2^{-2}}^\infty \frac{ds}{2^{s/2}\log (2+\frac{q}{\delta 2^s})} 
 &\simeq  \rho_2\int_{1}^\infty \frac{dt}{t^{3/2} \log (2+\frac{\log 1/\delta}{\delta t})}.
\end{align*}
Therefore, the proof of \eqref{eq:ldyadic} reduces to showing the bound
\begin{align}\label{eqn:1separate}
  \int_{1}^\infty \frac{dt}{t^{3/2} \log (2+\frac{\log 1/\delta}{\delta t})}  \lesssim (\log \frac{1}{\delta})^{-1}.
\end{align}
We begin with splitting the region of integration $ \int_{1}^{\infty} = \int_1^{(\log\frac{1}{\delta})^{2}} + \int_{(\log\frac{1}{\delta})^{2}}^\infty$.  The local part can be estimated as follows
\begin{align*}
    \int_{1}^{(\log\frac{1}{\delta})^{2}} \frac{dt}{t^{3/2} \log (2+\frac{\log 1/\delta}{\delta t})} 
    &\simeq \int_{1}^{(\log\frac{1}{\delta})^{2}} \frac{dt}{t^{3/2} \log (\frac{\log 1/\delta}{\delta t}) }\\
     &=\int_{1}^{(\log\frac{1}{\delta})^{2}} \frac{dt}{t^{3/2} ( \log\frac{1}{\delta} + \log\log{\frac{1}{\delta}} - \log t )}\\
     &\lesssim \int_{1}^{\infty} \frac{dt}{t^{3/2}  \log\frac{1}{\delta} }\\
     &\simeq (\log\frac{1}{\delta})^{-1},
     \end{align*}
where we used the fact that $( \log\frac{1}{\delta} + \log\log{\frac{1}{\delta}} - \log t ) \simeq \log \frac{1}{\delta}$ for $t\in \big(1,  (\log \frac{1}{\delta})^2\big)$.

For the global part we write
\begin{align*}
    \int_{(\log\frac{1}{\delta})^{2}}^{\infty} \frac{dt}{t^{3/2} \log (2+\frac{\log 1/\delta}{\delta t})} 
    \lesssim \int_{(\log\frac{1}{\delta})^{2}}^{\infty} \frac{dt}{t^{3/2} }
    \simeq (\log\frac{1}{\delta})^{-1},
\end{align*}
which shows \eqref{eqn:1separate} and consequently also \eqref{eq:ldyadic} and \eqref{eq:intermed}.

Combining \eqref{eqn:1supchange4} and \eqref{eq:intermed}, we conclude that
\begin{align*}
  \|  \sup_{\mathbf{y}\in \mathcal{F}, |A|\le m} \sum_{i\in A} \xi_i(\omega)y_i \|_{L^{q_0}(d\omega)}  &\lesssim m\rho_1 +\delta \rho_1^{-1} +q\rho_2(\log 1/ \delta)^{-1}+\rho_2^{-1}\\ &
  \lesssim \sqrt{\delta m} + (\log1/\delta)^{-1/2}(q_0+\log|\mathcal{F}|)^{1/2},
\end{align*}
    completing the proof of Lemma \ref{lem:1lem1}.
\end{proof}

We wish to underline and distill a simple version of one of the techniques used in the proof of Lemma \ref{lem:1lem1}, since it appears so many times at various points of Bourgain's overall argument.
\begin{prop}[Exchanging the supremum]\label{prop:sup-exchange}
    Suppose $n\ge 1$ and $\{X_j(\omega):j\in[n]\}$ is a collection of nonnegative random variables. Then for each $q\ge 1$,
    \[
    \int \sup_{j\in[n]}X_j(\omega)\,d\omega \le  n^{1/q}\cdot \sup_{j\in\mathcal [n]} (\int X_j(\omega)^q\,d\omega)^{1/q}.
    \]
    In particular, if $q \ge C^{-1}\log n$, then $\|\sup_{j\in\mathcal [n]} X_j(\omega)\|_{L^1(d\omega)} \le e^C \sup_{j\in\mathcal [n]}\|X_j(\omega)\|_{L^q(d\omega)}$.
\end{prop}

\section{Entropy estimates}\label{sec:entropy}
Denote by $|\cdot|$ the Euclidean norm on $\R^n$ and let $\|\cdot\|$ be some other norm on $\R^n$ with $X=(\R^n,\|\cdot\|)$. We will denote by $B^n$ the closed Euclidean unit ball and by $B_X$ the closed unit ball in norm $\|\cdot\|$.

For a set $U\subset \R^n$ the following quantities will be used to measure its entropy.

\begin{enumerate}
\item $E(U, B_X, t)= \inf \left\{N\ge 1: \exists \mathbf x_j \in \R^n, 1\le j \le N, U\subset \bigcup_{j=1}^N(\mathbf x_j+tB_X)\right\}$.  Note that here the centers of the balls are not necessarily in $U$.

\item $\widetilde{E}(U, B_X, t)= \inf \left\{N\ge 1: \exists \mathbf x_j \in U, 1\le j \le N, U\subset \bigcup_{j=1}^N(\mathbf x_j+tB_X)\right\}$.  Unlike in the definition of $E(U, B_X, t)$, the centers of the balls are required to lie in $U$.

\item $D(U,B_X,t) =  \sup\left\{M \ge 1: \exists \mathbf y_j\in U, 1\le j\le M, \|\mathbf y_j-\mathbf y_k\|> t, j\neq k \right\}$.
\end{enumerate}
\begin{rem}
    The number $E(U,B_X,t)$ agrees with the metric entropy number $N_d(U,t)$ we defined in Section \ref{sec:notation}, where the metric $d$ is induced by the norm $\|\cdot\|$ of $X$.
\end{rem}

The following result shows that the above quantities are comparable, which will allow us to conveniently choose whichever fits our application the best.

\begin{prop}\label{prop:equiv}
For any $U\subset \R^n$ and $t>0$ the following relations hold.
\begin{itemize}
\item[(a)] $D(U,B_X, t)\ge \widetilde{E}(U, B_X, t) \ge E(U, B_X, t) \ge D(U,B_X, 2t)$,

\item[(b)] $E(U, B_X, t) \ge \widetilde{E}(U,B_X, 2t)$.
\end{itemize}
\end{prop}

\begin{proof}

Let $M=D(U,B_X, t)$. Then there exists a sequence $\{\mathbf y_i\}_{i=1}^M\subset U$ such that $\|\mathbf y_j-\mathbf y_k\|> t$ for $j\neq k$ and the collection of balls $\{\mathbf y_j+tB_X\}_{j=1}^M$ covers $U$, by maximality of $M$. Since $\widetilde{E}(U, B_X, t)$ is the infimum over all such collections, it follows that $D(U,B_X, t)\ge \widetilde{E}(U, B_X, t)$.

The inequality $\widetilde{E}(U, B_X, t) \ge E(U, B_X, t)$ is immediate from the definitions. 

To show that $E(U, B_X, t) \ge D(U,B_X, 2t)$ take any cover $U\subset \bigcup_{j=1}^N \{ \mathbf x_j +tB_X\}$ and note that if $\{\mathbf y_k\}_{k=1}^M$ are such that  $\|\mathbf y_j-\mathbf y_k\|> 2t$ for $j\neq k$, then each $\mathbf y_k$ lies in at most one ball from  $\{ \mathbf x_j +tB_X\}_{j=1}^N$. Therefore $N \ge M$ and taking the infimum over $N\ge 1$ gives the postulated bound.

Finally, to show part (b) consider a cover $U\subset \bigcup_{j=1}^N \{ \mathbf x_j +tB_X\}$ with $\mathbf x_j\in\R^n$ and assume that for each $j$ there exists $\mathbf u_j\in U \cap \{ \mathbf x_j +tB_X\}$ (if such $\mathbf u_j$ does not exist it means that a ball $\mathbf x_j +tB_X$ is redundant and can be removed from our cover). Then by the triangle inequality, $\{\mathbf u_j+2tB_X\}\supset \{\mathbf x_j+tB_X\}$ for all $j \in [N]$ and consequently $E(U, B_X, t) \ge \widetilde{E}(U,B_X, 2t)$.
\end{proof}

Our next goal will be to obtain suitable estimates for the entropy number $E(B^n, B_X, t)$. We begin with a simple observation based on volume counting.

\begin{prop}\label{prop:vol}
Let $\|\cdot\|$ be a norm in $\R^n$ with a unit ball $B_X$. Then 
$$
D(B_X,B_X, t) \le (4/t)^n, \qquad 0< t < 1.
$$
\end{prop}

\begin{proof}
Let $M=D(B_X,B_X, t)$. There exist $\{\mathbf y_j\}_{j=1}^M\subset B_X$ such that  $\|\mathbf y_j-\mathbf y_k\|> t$ for $j\neq k$. Then the balls $\{\mathbf y_j+\frac{1}2 t B_X\}_{j=1}^M$ are pairwise disjoint. Moreover, $\mathbf y_j+\frac{1}2 t B_X\subset 2B_X$ for each $j$, since $t<1$. Combining these two observations we get
\begin{align*}
\sum_{j=1}^M \big|\frac{1}2 t B_X\big| = \big|\sum_{j=1}^M (\mathbf y_j+\frac{1}2 t B_X)\big| \le |2B_X|,
\end{align*} 
where $|\cdot|$ stands for the Lebesgue measure. 

Thus,
\begin{equation*}
M\Big(\frac{1}2 t\Big)^n |B_X| \le 2^n |B_X|,
\end{equation*}
and finally
$$
M \le (4/t)^n. 
$$
\end{proof}

The upper bound for the entropy will be expressed in terms of a quantity called a L\'evy mean. 

\begin{definition}
The \emph{L\'evy mean} of the normed space  $X = (\mathbb R^n,\|\cdot\|)$ is given by
\begin{equation*}
M_X := \int_{\mathbb{S}^{n-1}}\|\mathbf{u}\|\,d\sigma(\mathbf{u}),
\end{equation*}
where $d\sigma$ denotes the normalized surface measure on the Euclidean unit sphere $\mathbb{S}^{n-1}$.
\end{definition}

We will estimate $E(B^n, B_X, t)$ by an expression involving $M_X$. We remark that any two norms in $\R^n$ are comparable, but only up to multiplicative constants which may depend on $n$ (for example one has $\|\cdot\|_{\ell^2} \le \|\cdot\|_{\ell^1} \le \sqrt{n}\|\cdot\|_{\ell^2}$). That dependence is critical for our developments since we will consider large values of $n$.

We begin with a result collecting several representations of the L\'evy mean.

\begin{prop}\label{prop:mx}
Let $\{g_i\}_{i=1}^n$ be i.i.d. standard normal random variables, let $\{\mathbf{e}_i\}_{i=1}^n$ be the standard orthonormal basis in $\R^n$, and let $X=(\R^n,\|\cdot\|)$ for some norm $\|\cdot\|$. Then the following identities hold:
\begin{equation*}
M_X=\alpha_n (2\pi)^{-n/2}\int_{\R^n}e^{-|\mathbf{x}|^2/2}\|\mathbf{x}\|\,d\mathbf{x}=\alpha_n\int_{\Omega}\Big\|\sum_{i=1}^n g_i(\omega)\mathbf{e}_i\Big\|\,d\omega,
\end{equation*}
where 
$$
\alpha_n=\frac{\Gamma(\frac{n}2)}{\sqrt{2}\Gamma(\frac{n+1}2)}\simeq \frac{1}{\sqrt{n}}.
$$
\end{prop}

\begin{proof}
Let $d\widetilde{\sigma}$ denote the (non-normalized) surface measure on $\mathbb{S}^{n-1}$. Using polar coordinates one gets
\begin{align*}
\int_{\R^n}e^{-|\mathbf{x}|^2/2}\|\mathbf{x}\|\,d\mathbf{x}&=\int_{0}^\infty e^{-r^2/2}r^{n}\,dr \int_{\mathbb{S}^{n-1}}\|\mathbf{u}\|\,d\widetilde{\sigma}(\mathbf{u})
\\
&=d\widetilde{\sigma}(\mathbb{S}^{n-1})\int_{0}^\infty e^{-r^2/2}r^{n}\,dr \int_{\mathbb{S}^{n-1}}\|\mathbf{u}\|\,d\sigma(\mathbf{u}).
\end{align*}
Thus, the first identity in the statement of the proposition will follow if we show
\begin{equation*}
d\widetilde{\sigma}(\mathbb{S}^{n-1})\int_{0}^\infty e^{-r^2/2}r^{n}\,dr =\frac{(2\pi)^{n/2}}{\alpha_n}.
\end{equation*}
To verify the above it suffices to combine 
$$
d\widetilde{\sigma}(\mathbb{S}^{n-1})=\frac{2\pi^{n/2}}{\Gamma(n/2)}
$$
with 
$$
\int_{0}^\infty e^{-r^2/2}r^{n}\,dr=2^{\frac{n-1}2}\Gamma\big(\frac{n+1}2\big).
$$
The latter requires a simple change of variables. We omit the details.

To complete the proof it remains to show
\begin{equation*}
\int_{\Omega}\big\|\sum_{i=1}^n g_i(\omega)\mathbf{e}_i\big\|\,d\omega =(2\pi)^{-n/2}\int_{\R^n}e^{-|\mathbf{x}|^2/2}\|\mathbf{x}\|\,d\mathbf{x}.
\end{equation*}
This however follows immediately from the fact that due to the independence of the random variables $\{g_i\}_{i=1}^n$, the density of the random vector $(g_1(\omega), \dots, g_n(\omega))=\sum_{i=1}^n g_i(\omega)\mathbf{e}_i$ is simply a product of 1-dimensional densities of the form $(2\pi)^{-1/2}e^{-x_j^2/2}$, for $j\in [n]$.
\end{proof}

Now we are ready to show the following important estimate for the metric entropy of the Euclidean unit ball $B^n$ in the normed space $X$ in terms of the L\'evy mean $M_X$.

\begin{prop}\label{prop:key}
For any $t>0$ and $n\in \N$ we have
\begin{equation*}
\log E(B^n, B_X, t) \lesssim n \big(\frac{M_X}t\big)^2,
\end{equation*}
with the implicit constant independent of $t$ and $n$.
\end{prop}

\begin{proof}
Let $N=D(B^n, B_X, t)$. Then there exists a maximal sequence $\{\mathbf{x}^{(i)}\}_{i=1}^N\subset B^n$, such that $\|\mathbf{x}^{(j)}-\mathbf{x}^{(k)}\|> t$ for $j\neq k$. By Proposition \ref{prop:equiv} it suffices to prove that
$$\log N \lesssim n \big(\frac{M_X}t\big)^2.$$
To this end we will construct a probability space and a collection of $N$ disjoint balls inside it, each of them with a relatively large measure.

Consider a probability measure on $\R^n$ given by
\begin{equation*}
d\mu(\mathbf{x})=(2\pi)^{-n/2}e^{-|\mathbf{x}|^2/2}\,d\mathbf{x}.
\end{equation*}
By the first identity in Proposition \ref{prop:mx} we have 
\begin{align*}
\alpha_n^{-1}M_X=\int_{\R^n}\|\mathbf{x}\|\,d\mu(\mathbf{x})=\mathbb{E}_{\mu}(\|\mathbf{x}\|).
\end{align*}
Now let $R:=2\mathbb{E}_{\mu}(\|\mathbf{x}\|)$ and notice that as a consequence of Chebyshev's inequality we have
\begin{equation}\label{eq:muball}
\mu(RB_X)=\mu(\|\mathbf{x}\|\le R)\ge 1/2.
\end{equation}
Clearly the balls $\{\mathbf{x}^{(j)}+\frac{1}2 t B_X\}_{j=1}^N$ are mutually disjoint. We will rescale them to get balls of radius $R$. To this end let
$\mathbf{y}^{(j)}=2Rt^{-1}\mathbf{x}^{(j)}$ and note that the balls $\{\mathbf{y}^{(j)}+RB_X\}_{j=1}^N$ are also mutually disjoint. Using the symmetry property $B_X = -B_X$, and later the convexity of the function $e^{-u}$ and \eqref{eq:muball}, we obtain 
\begin{align*}
\mu(\mathbf{y}^{(i)}+RB_X)&=(2\pi)^{-n/2}\int_{RB_X} e^{-|\mathbf{y}-\mathbf{y}^{(i)}|^2/2}\,d\mathbf y
\\
&=(2\pi)^{-n/2}\int_{RB_X} \frac{1}{2}\left( e^{-|\mathbf{y}-\mathbf{y}^{(i)}|^2/2}+e^{-|\mathbf{y}+\mathbf{y}^{(i)}|^2/2}\right)\,d\mathbf y
\\
&\ge (2\pi)^{-n/2}\int_{RB_X} e^{-(|\mathbf{y}-\mathbf{y}^{(i)}|^2+|\mathbf{y}+\mathbf{y}^{(i)}|^2)/4}\,d\mathbf y
\\
&=(2\pi)^{-n/2}\int_{RB_X} e^{-(|\mathbf{y}|^2+|\mathbf{y}^{(i)}|^2)/2}\,d\mathbf y
\\
&=\mu(RB_X) e^{-|\mathbf{y}^{(i)}|^2/2}\ge \frac{1}2 e^{-|\mathbf{y}^{(i)}|^2/2}.
\end{align*}
Now, since $\mathbf{x}^{(i)}\in B^n$, we have $|\mathbf{y}^{(i)}|\le 2Rt^{-1}$, so the estimate derived above gives
\begin{equation*}
\mu(\mathbf{y}^{(i)}+RB_X)\ge \frac{1}2 e^{-(2Rt^{-1})^2/2}.
\end{equation*}
Hence, due to the disjointness of balls $\{\mathbf{y}^{(j)}+RB_X\}_{j=1}^N$ we get
\begin{align*}
1=\mu(\R^n)\ge \sum_{i=1}^N \mu(\mathbf{y}^{(i)}+RB_X)\ge \frac{N}{2} e^{-(2Rt^{-1})^2/2},
\end{align*}
which is equivalent to 
\begin{equation*}
N< 2 e^{(2Rt^{-1})^2/2}.
\end{equation*}
Applying $\log$ to both sides, using the definition of $R$ and the fact that $\alpha_n\simeq \frac{1}{\sqrt{n}}$ we get
\begin{align*}
\log N\lesssim (Rt^{-1})^2\simeq M_X^2 \alpha_n^{-2} t^{-2}\simeq  n \left(\frac{M_X}{t}\right)^2,
\end{align*}
and the proof is finished.
\end{proof}

In what follows we will identify $\R^n$ with the $\mathbb R$-linear span of $n$ mutually orthogonal and $1$-bounded functions $\varphi_1(u),\dots,\varphi_n(u)$ in the $L^2$ inner product 
\[
\langle f,g\rangle = \int f(u)g(u)\,du
\]
\begin{equation}\label{eq:iden}
\R^n\cong \Big\{\sum_{j=1}^n a_j \varphi_j(u): a_j\in \R\Big\}.
\end{equation}
Here as before, we follow Bourgain's convention of suppressing the domain of the functions $\varphi_j$ as well as the particular probability measure from the notation. (See the discussion at the beginning of Section \ref{sec:dec} for more about Bourgain's notational convention.)
Note that $\|\sum_{j=1}^n a_j \varphi_j\|_{L^2(du)}\le |(a_1,\dots,a_n)|$ and for each $j$, the standard basis vector $\mathbf{e}_j$ of $\R^n$ corresponds to $\varphi_j$.
With this identification, for a fixed $q\ge 2$ we will consider $X=(\R^n,\|\cdot\|_{L^q(du)})$. In analogy with the case where $\varphi_j(\theta) = e^{2\pi ij\theta}$ on the circle, we will sometimes refer to the indices $j$ as \emph{frequencies}.
\begin{rem}
    In the rest of the paper, we will assume that the functions $\varphi_j$ are mutually orthogonal and $1$-bounded. For  concreteness, the reader may like to take the functions $(\sin(2\pi j\theta))_{j=1}^n$ on $(\mathbb T,d\theta)$.
\end{rem}

The following lemma shows that the L\'evy mean for the space $X$ is controlled by $\sqrt{q}$. 

\begin{lem}\label{lem:mxlq}
For $q\ge 2$ let $X=(\R^n,\|\cdot\|_{L^q(du)})$. Then the following estimate holds
$$
M_X\lesssim_q 1,
$$ 
with the implicit constant independent of $n$.
\end{lem}

\begin{rem}
The proof below shows in fact that 
$$
M_X\lesssim \sqrt{q},
$$ 
with the implicit constant independent of $n$ and $q$. However, the exact dependence on $q$ will not be relevant for our application.
\end{rem}

\begin{proof}
Denote by $\{r_j(t)\}_{j=1}^n$ a system of Rademacher functions, that represents independent random choices of sign. Note that if $\{g_j(\omega)\}_{j=1}^n$ are i.i.d. standard normal variables, then for each fixed $t\in (0,1)$, the random variables  $\{r_j(t)g_j(\omega)\}_{j=1}^n$ are also  i.i.d. standard normal variables.
 
Using Proposition \ref{prop:mx} together with the above observation we get for any $t\in(0,1)$
\begin{align*}
M_X=\alpha_n\int_{\Omega}\Big\|\sum_{j=1}^n g_j(\omega)\varphi_j(u)\Big\|_{L^q(du)}\,d\omega=\alpha_n\mathbb{E}_t\int_{\Omega}\Big\|\sum_{j=1}^n r_j(t)g_j(\omega)\varphi_j(u)\Big\|_{L^q(du)}\,d\omega.
\end{align*}
Next using Fubini--Tonelli's theorem and later H\"older's inequality we can write
\begin{align*}
M_X&=\alpha_n\int_{\Omega}\mathbb{E}_t\Big\|\sum_{j=1}^n r_j(t)g_j(\omega)\varphi_j(u)\Big\|_{L^q(du)}\,d\omega
\\
&\le \alpha_n\int_{\Omega}\Big(\mathbb{E}_t\Big\|\sum_{j=1}^n r_j(t)g_j(\omega)\varphi_j(u)\Big\|^q_{L^q(du)}\Big)^{1/q}\,d\omega
\\
&=\alpha_n\int_{\Omega}\Big(\int \int_0^1 |\sum_{j=1}^n r_j(t)g_j(\omega)\varphi_j(u)|^q\, dt\, du\Big)^{1/q}\,d\omega.
\end{align*}
To estimate the latter expression we apply Khintchine's inequality (to the integral in $dt$), noting that the constant in this inequality is of size $\sqrt{q}$, and then use 1-boundedness of the functions $\varphi_j$ to get
\begin{align*}
M_X&\lesssim \alpha_n\int_\Omega\Big(\int\big(\sum_{j=1}^n|g_j(\omega)\varphi_j(u)|^2\big)^{q/2}\,du\Big)^{1/q}\,d\omega
\\
&\lesssim\alpha_n \int_{\Omega}\Big( \sum_{j=1}^n |g_j(\omega)|^2\Big)^{1/2}\, d\omega.
\end{align*}
Finally, we apply the Cauchy--Schwarz inequality in the integral over $\omega$ and later use the fact that $\int_\Omega g_j^2(\omega)\,d\omega=1$ getting
\begin{align*}
M_X&\lesssim_q\alpha_n\Big(\int_{\Omega} \sum_{j=1}^n |g_j(\omega)|^2 \,d\omega\Big)^{1/2}=\alpha_n \sqrt{n}\simeq 1,
\end{align*}
where the last relation follows from the estimate $\alpha_n\simeq\frac{1}{\sqrt{n}}$.
\end{proof}

Combining Proposition \ref{prop:key} with Lemma \ref{lem:mxlq} we get the following result.
\begin{cor}\label{cor:key}
For any $t>0$ and $n,q \in \N$ we have
\begin{equation}\label{eq:235}
\log E(B^n, B_{L^q(du)}, t) \lesssim_q nt^{-2}. 
\end{equation}
\end{cor}
The above estimate will be a key tool in estimating the entropy in this paper. However, as we will soon see, in order to apply it efficiently, we will have to set the stage suitably by ``reducing'' the dimension of the ball $B^n$.  

The following is the set whose entropy will be the object of our study in the rest of this section.
\begin{definition}
For $m\in [n]$ let
\begin{equation*}
\mathcal{P}_m = \Big\{\sum_{j\in A} a_j \varphi_j(u): a_j\in \R,\quad |(a_1, \dots, a_n)|\le 1,\quad |A|=m\Big\}.
\end{equation*}
\end{definition}
Observe that with the identification \eqref{eq:iden} we can view $\mathcal{P}_m$ as the collection of all points lying at the intersection of the Euclidean unit ball $B^n$ with some $m$-dimensional subspace of $\R^n$ spanned by coordinate vectors. We emphasize that $\mathcal{P}_m\ncong \R^m$; the set $\mathcal{P}_m$  is much larger. In fact we have
\begin{equation}\label{eq:pmdec}
\mathcal{P}_m = \bigcup_{\substack{A\subset [n] \\ |A|=m}}\mathcal{P}_A,
\end{equation}
where 
$$
\mathcal{P}_A=\Big\{\sum_{j\in A} a_j \varphi_j(u): a_j\in \R,\ |(a_1, \dots, a_n)|\le 1\Big\}.
$$

The key result of this section is the estimate of 
$$
N_q(\mathcal{P}_m, t):=E(\mathcal{P}_m, B_{L^q(du)}, t).
$$ 

\begin{thm}\label{thm:entropymain}
For any $m\le n$ and $2\le q \le \infty$ the following bounds hold
\begin{align}\label{eq:tlarge}
\log N_q(\mathcal{P}_m, t) &\lesssim m t^{-\nu} \log\Big(1+\frac{n}{m}\Big), \qquad t\ge8,
\\
\label{eq:tsmall}
\log N_q(\mathcal{P}_m, t) &\lesssim m \log(1+\frac{1}t) \log\Big(1+\frac{n}{m}\Big), \qquad 0 <  t < 8,
\end{align}
with the implicit constant depending only on $q$ and $\nu=\nu(q)>2$.
\end{thm}

Note that for a single set $A\subset[n]$ of cardinality $m$ we have $\mathcal{P}_A\cong B^m$ and consequently, the application of Corollary \ref{cor:key} gives
$$
\log N_q(\mathcal{P}_A, t):= \log E(\mathcal{P}_A, B_{L^q(du)}, t)= \log E(B^m, B_{L^q(du)}, t) \lesssim_q mt^{-2}.
$$   
On the other hand, we clearly have 
$
\mathcal{P}_m \subset \Big\{\sum_{j=1}^n a_j \varphi_j(u): a_j\in \R, |(a_1, \dots, a_n)|\le 1 \Big\}\cong B^n, 
$
so we can use Corollary \ref{cor:key} again to get
$$
\log N_q(\mathcal{P}_m, t)\le \log E(B^n, B_{L^q(du)}, t) \lesssim_q nt^{-2}.
$$   
This bound, however, is far from optimal when $n$ is large. Theorem \ref{thm:entropymain} shows that one can improve it substantially.

Furthermore, notice that for $t\simeq 1$ the right-hand sides of \eqref{eq:tlarge} and \eqref{eq:tsmall} are comparable. Therefore, there is nothing special about choosing $8$ to separate the estimates for small and large values of $t$; we can replace the range $t\ge 8$ by $t \ge a$ for any $a>0$, which will only affect a multiplicative constant.

Finally, we remark that, as we will soon see, the estimate \eqref{eq:tsmall} can be obtained from \eqref{eq:tlarge} via a standard covering argument and simple volume counting, so the whole difficulty lies in handling large values of $t$.

We will obtain Theorem \ref{thm:entropymain} as a corollary of the following result.

\begin{thm}\label{thm:entropymain2}
For any $m\in [n]$ and $2\le r \le \infty$ the following bound holds
\begin{equation}\label{eq:2.5}
\log N_r(\mathcal{P}_m, t) \lesssim m t^{-2}(\log t) \log\Big(1+\frac{n}{m}\Big), \qquad t > 2,
\end{equation}
with the implicit constant depending only on $r$.
\end{thm}

\begin{prop}\label{prop:impl}
Theorem \ref{thm:entropymain2} implies Theorem \ref{thm:entropymain}.
\end{prop}

\begin{proof}
First we show that if \eqref{eq:2.5} holds for all $r\ge 2$ and all $t > 2$, then for any $q \ge 2$ and $t\ge 8$  the estimate \eqref{eq:tlarge} holds with some $\nu>2$.

To this end fix $q\ge 2$ and choose any $r>q$. Let $\theta\in(0,1)$ be such that 
$$
\frac{1}q=\frac{1-\theta}2+\frac{\theta}r.
$$
For any $f,g\in\mathcal{P}_m$ we have by H\"older's inequality
\begin{equation*}
\|f-g\|_{q}\le \|f-g\|_{2}^{1-\theta}\|f-g\|_{r}^\theta\le 2\|f-g\|_r^\theta,
\end{equation*}
where the last estimate holds since $\|f-g\|_{2}\le \|f\|_2+\|g\|_{2}\le 2$. Therefore, if $\{f_i\}_{i=1}^n$ are such that $\|f_j-f_k\|_q\ge t$ for $j\neq k$, then also $\|f_j-f_k\|_r\ge (\frac{t}2)^{1/\theta}$ for $j\neq k$. Thus, in view of Proposition \ref{prop:equiv}, we get for $t\ge 8$
\begin{align*}
\log N_q(\mathcal{P}_m, t)\le \log D(\mathcal{P}_m, B_{L^q(du)}, t) \le \log D(\mathcal{P}_m, B_{L^r(du)}, \big(\frac{t}2\big)^{1/\theta}) \le \log N_r\big(\mathcal{P}_m, \frac{1}{2} \big(\frac{t}2\big)^{1/\theta}\big).
\end{align*}
Now, since $\frac{1}{2}\big(\frac{t}2\big)^{1/\theta} > 2$ for $t\ge 8$ and $\theta\in (0,1)$, we can use \eqref{eq:2.5} to write
\begin{align*}
\log N_r\big(\mathcal{P}_m, \frac{1}{2} \big(\frac{t}2\big)^{1/\theta}\big)&\lesssim m \Big(\frac{t}2\Big)^{-2/\theta}\log \Big(\frac{t}2\Big)^{1/\theta}\log\Big(1+\frac{n}{m}\Big)\simeq
m t^{-2/\theta}\log t \log\Big(1+\frac{n}{m}\Big)
\\
&\lesssim m t^{-\nu} \log\Big(1+\frac{n}{m}\Big),
\end{align*}
for some $\nu>2$, since $-\frac{2}\theta < -2$.

It remains to show that \eqref{eq:tlarge} implies \eqref{eq:tsmall}.  Let $t < 4$ and notice that due to \eqref{eq:pmdec} we can write
\begin{align*}
 N_q(\mathcal{P}_m, t)\le \#\{A\subset [n]: |A|=m\} \sup_{|A|=m} N_q(\mathcal{P}_A, t)=\binom{n}{m}\sup_{|A|=m} N_q(\mathcal{P}_A, t).
\end{align*}
Applying $\log$ to both sides and using the estimate $(\frac{n}{m})^m \le \binom{n}{m} \le (\frac{en}{m})^m$ we get
\begin{align}\label{eq:logged}
\log N_q(\mathcal{P}_m, t)\le \log \binom{n}{m} + \sup_{|A|=m} \log N_q(\mathcal{P}_A, t)\le 
m\log(1+\frac{n}m) + \sup_{|A|=m} \log N_q(\mathcal{P}_A, t).
\end{align}
Next notice that for each set $A$ we can dominate $N_q(\mathcal{P}_A, t)$ by covering $\mathcal{P}_A$ with $\|\cdot\|_q$-balls of radius 4 and then find the minimum number of $\|\cdot\|_q$-balls of radius $t$ needed to cover each of these balls. That leads to the estimate
\begin{equation*}
N_q(\mathcal{P}_A, t)\lesssim D(4B_{L^q(du)}, B_{L^q(du)}, t)N_q(\mathcal{P}_A, 4)=D(B_{L^q(du)}, B_{L^q(du)}, t/4)N_q(\mathcal{P}_A, 4).
\end{equation*}
The last identity above is just rescaling. Using Proposition \ref{prop:vol} and \eqref{eq:tlarge} (with $t=4$) we get
\begin{align*}
\sup_{|A|=m} \log N_q(\mathcal{P}_A, t)\le \sup_{|A|=m} \log\Big( \Big(\frac{16}t\Big)^m N_q(\mathcal{P}_A, 4)\Big)\le m\log\Big(1+\frac{1}t\Big) + m  \log\Big(1+\frac{n}{m}\Big).
\end{align*}
Plugging the above bound to \eqref{eq:logged} and noting that $A+B\lesssim AB$, for $A,B\gtrsim 1$, concludes the proof.
\end{proof}

In view of Proposition \ref{prop:impl}, proving Theorem \ref{thm:entropymain} reduces to showing Theorem \ref{thm:entropymain2}.
Before we proceed with the proof we need the following simple version of the pigeonhole principle.

\begin{lem}\label{lem:inc-ex}
Let $X_1, X_2$ and $X_3$ be positive random variables on the probability space $(\Omega, \mathbb{P})$ such that 
$$
\mathbb{E}X_i\le m_i,\qquad i\in[3],
$$
for some positive constants $m_1,m_2$ and $m_3$. Then 
\begin{equation*}
\mathbb{P}(\{\omega\in\Omega: X_1(\omega)< 10 m_1 \quad\textrm{and}\quad X_2(\omega)< 10 m_2 \quad\textrm{and}\quad X_3(\omega)< 10 m_3\}) \ge \frac{3}{10}.
\end{equation*}
In particular, there exists $\omega_0\in \Omega$ such that 
$$
X_1(\omega_0)< 10 m_1 \quad\textrm{and}\quad X_2(\omega_0)< 10 m_2 \quad\textrm{and}\quad X_3(\omega_0)< 10 m_3.
$$
\end{lem}
\begin{proof}
Let $A_i = \{\omega\in\Omega: X_i(\omega)< 10 m_i\}$, $i\in[3]$. Note that by Chebyshev's inequality we have  
\begin{equation*}
\mathbb{P}(A_i)=1-\mathbb{P}(\{\omega\in\Omega: X_i(\omega)\ge 10 m_i\})\ge 1- \frac{\mathbb{E}X_i}{10m_i}\ge \frac{9}{10}, \qquad i\in[3].
\end{equation*}
Applying the inclusion-exclusion principle and using the above estimate we get for any $i\neq j$
\begin{equation*}
\mathbb{P}(A_i\cap A_j)=\mathbb{P}(A_i)+\mathbb{P}(A_j)-\mathbb{P}(A_i\cup A_j)\ge \frac{9}{10}+\frac{9}{10}-1=\frac{8}{10}.
\end{equation*}
Finally, another application of the inclusion-exclusion principle gives
\begin{align*}
\mathbb{P}(A_1\cap A_2\cap A_3)&=\mathbb{P}(A_1\cup A_2\cup A_3)+\mathbb{P}(A_1\cap A_2)+\mathbb{P}(A_1\cap A_3)+\mathbb{P}(A_2\cap A_3)
\\
&\quad -\mathbb{P}(A_1)-\mathbb{P}(A_2)-\mathbb{P}(A_3)
\\
&\ge \frac{9}{10}+3\frac{8}{10}-3=\frac{3}{10}.
\end{align*}
\end{proof}
\begin{proof}[Proof of Theorem \ref{thm:entropymain2}]
Fix $t>2$ and let $k\in \N$ be such that $2^{k/2}\le t < 2^{1+k/2}$. For a fixed $A\subset[n]$ with $|A|=m$ and $|a|=1$ let $f=\sum_{i\in A} a_i \varphi_i$. Denoting by $\{r_j\}_{j=1}^n$ a system of independent Rademacher functions on the interval $(0,1)$  we let for $i\in[n]$ and $j\in [k]$
$$
\varepsilon_i^j:= r_i(\omega_j), \qquad \omega=(\omega_1,\dots, \omega_k)\in(0,1)^k.
$$ 
Then $\mathbf{\varepsilon}:=(\varepsilon_i^j)_{1\le i\le n, 1\le j\le k}$ forms a collection of randomly selected, independent elements of the set $\{-1,1\}$. We will write $d\varepsilon^j :=d\omega_j$ and $d\varepsilon:=d\varepsilon^1\dotsb d\varepsilon^k$.

For $i\in [n]$ consider the decomposition 
\begin{align*}
1&=\varepsilon_i^1+(1-\varepsilon_i^1)=\varepsilon_i^1+(1-\varepsilon_i^1)\big(\varepsilon_i^2+(1-\varepsilon_i^2)\big)=\varepsilon_i^1+(1-\varepsilon_i^1)\varepsilon_i^2+(1-\varepsilon_i^1)(1-\varepsilon_i^2)=\dots
\\
&=\varepsilon_i^1+(1-\varepsilon_i^1)\varepsilon_i^2+(1-\varepsilon_i^1)(1-\varepsilon_i^2)+\dots+(1-\varepsilon_i^1)\dots(1-\varepsilon_i^{k-1})\varepsilon_i^k+(1-\varepsilon_i^1)\dots(1-\varepsilon_i^k).
\end{align*}
Using the above decomposition for each $i\in A$, we can represent $f$ as follows
\begin{align*}
f(u)=\sum_{i\in A} a_i \varphi_i(u)= \Phi(\mathbf{\varepsilon},u)+E(\mathbf{\varepsilon},u),
\end{align*}
where
\begin{equation*}
\Phi(\mathbf{\varepsilon},u)=\sum_{i\in A} a_i \varepsilon_i^1\varphi_i(u)+\sum_{i\in A} a_i (1-\varepsilon_i^1)\varepsilon_i^2\varphi_i(u)+\dots+\sum_{i\in A} a_i (1-\varepsilon_i^1)\dots(1-\varepsilon_i^{k-1})\varepsilon_i^k\varphi_i(u)
\end{equation*}
and
\begin{equation*}
E(\mathbf{\varepsilon},u)=\sum_{i\in A_\varepsilon} a_i (1-\varepsilon_i^1)\dots(1-\varepsilon_i^k)\varphi_i(u),
\end{equation*}
with 
$$
A_\varepsilon=\{i\in A: \varepsilon_i^1=\dots=\varepsilon_i^k=-1\}.
$$
Note that the range of summation in the definition of $E(\mathbf{\varepsilon},u)$ can be restricted to $i\in A_\varepsilon$, since the summands $a_i (1-\varepsilon_i^1)\dots(1-\varepsilon_i^k)\varphi_i(u)$ vanish for $i\in A\setminus A_\varepsilon$.

We will show that for any fixed $u$ the following are true

 \begin{equation}\label{eq:Phi}
\int\|\Phi(\mathbf{\varepsilon},u)\|_{L^q(du)}\,d\varepsilon \lesssim t,
\end{equation}
 \begin{equation}\label{eq:Aeps}
\int |A_\varepsilon|\,d\varepsilon \lesssim \frac{m}{t^2},
\end{equation}
and 
 \begin{equation}\label{eq:E}
\int\|E(\mathbf{\varepsilon},u)\|_{L^2(du)}\,d\varepsilon \lesssim t.
\end{equation}
Assume momentarily that \eqref{eq:Phi}, \eqref{eq:Aeps} and \eqref{eq:E} hold. Then by Lemma \ref{lem:inc-ex} (applied with $X_1=\|\Phi(\mathbf{\varepsilon},u)\|_{L^q( du)},X_2=|A_\varepsilon|$ and $X_3=\|E(\mathbf{\varepsilon},u)\|_{L^2(du)}$) there exists $\varepsilon_0=(\varepsilon_i^j)_{1\le i\le n, 1\le j\le k}$ such that 
 \begin{equation}\label{eq:Phi0}
\|\Phi(\varepsilon_0,u)\|_{L^q( du)} \lesssim t,
\end{equation}
 \begin{equation}\label{eq:Aeps0}
 |A_{\varepsilon_0}| \lesssim \frac{m}{t^2},
\end{equation}
and 
 \begin{equation}\label{eq:E0}
\|E(\varepsilon_0,u)\|_{L^2(du)} \lesssim t.
\end{equation}
In other words, 
\begin{equation}\label{eq:approx}
\|f(u)-E(\varepsilon_0,u)\|_{L^q( du)}\le c t,
\end{equation}
with $E(\varepsilon_0,\cdot)\in ct\mathcal{P}_{\lfloor m/t^2\rfloor}$ and some absolute constant $c>0$.
 
The equation \eqref{eq:approx} means that any function $f\in \mathcal{P}_m$ can be approximated in the norm $\|\cdot\|_{L^q(du)}$ within $t$ by some function $E(\varepsilon_0,\cdot)$ from $ct\mathcal P_{\lfloor m/t^2\rfloor}$. Note that for large values of $t$ the frequency support of each function in $ct\mathcal{P}_{\lfloor m/t^2\rfloor}$ is substantially smaller than $m$---the size of the frequency support of $f$. That gain plays a crucial role in the proof. Bourgain refers to the procedure described above as the ``method of support-reduction.''

In view of Proposition \ref{prop:equiv} (consider a maximal $ct$-separated subset of $\mathcal{P}_m$), we get 
\begin{align*}
N_q(\mathcal{P}_m, ct)\lesssim N_q(ct\mathcal{P}_{\lfloor m/t^2\rfloor}, ct)=N_q(\mathcal{P}_{\lfloor m/t^2\rfloor}, 1),
\end{align*}
with the same $c$ as in \eqref{eq:approx}. The last identity follows simply from rescaling. Now using \eqref{eq:logged} we obtain
\begin{equation*}
N_q(\mathcal{P}_{\lfloor m/t^2\rfloor}, 1)\lesssim \log{ \binom{n}{\lfloor m/t^2\rfloor}} + \sup_{|A|= \lfloor m/t^2\rfloor} \log N_q(\mathcal{P}_A, 1).
\end{equation*}
Using Stirling's formula we can estimate
\begin{align*}
\log {\binom{n}{\lfloor m/t^2\rfloor}}\lesssim \log {\Big(\frac{e n}{m/t^2}\Big)^{m/t^2}}\simeq \frac{m}{t^2} \log \Big(\frac{t^2 n}{m}\Big) \simeq \frac{m}{t^2}\Big(\log t+\log \frac{n}m\Big).
\end{align*}
Let us emphasize, that the extra decay $t^{-2}$ in the above bound is due to the ``support-reduction'' procedure applied before.
Moreover, for any $A$ of size $\lfloor m/t^2\rfloor$ we can identify $\mathcal{P}_A\cong (\R^{\lfloor m/t^2\rfloor}, \|\cdot\|_{L^q(du)})$ and consequently by \eqref{eq:235} the bound
\begin{equation*}
\log N_q(\mathcal{P}_A, t)=\log E(B^{\lfloor m/t^2\rfloor}, B_{L^q(du)}, 1) \lesssim m t^{-2}
\end{equation*} 
holds uniformly in $A\subset [n]$ satisfying $|A|= \lfloor m/t^2\rfloor$.

Combining the above estimates we obtain
\begin{equation*}
N_q(\mathcal{P}_{\lfloor m/t^2\rfloor}, 1)\lesssim \frac{m}{t^2}\Big(\log t+\log \frac{n}m\Big)+ m t^{-2}\lesssim mt^{-2}(\log t) \log {(\frac{n}{m}+1)},
\end{equation*}
where in the last relation we used the fact that product of terms greater than 2 dominates their sum. That gives the postulated bound. 

It remains to prove \eqref{eq:Phi}, \eqref{eq:Aeps} and \eqref{eq:E}. Let us begin with \eqref{eq:Phi}. By H\"older's inequality (applied to the integral with respect to $d\varepsilon^l$) and Khintchine's inequality (applied again to the integration with respect to $d\varepsilon^l$; the constant in that inequality is of size $\sqrt{q}$) we get
\begin{align*}
&\int\|\Phi(\mathbf{\varepsilon},u)\|_{L^q( du)}\,d\varepsilon
 \\
&\quad\le \sum_{l=1}^k \int \big\|\|\sum_{i\in A} a_i (1-\varepsilon_i^1)\dotsb(1-\varepsilon_i^{l-1})\varepsilon_i^l\varphi_i(u)\|_{L^q(d\varepsilon^l)}\big\|_{L^q(du)}\, d\varepsilon^1\dotsb d\varepsilon^{l-1}
\\
&\quad\lesssim_q \sum_{l=1}^k  \int \Big[\sum_{i\in A} |a_i|^2 (1-\varepsilon_i^1)^2\dotsb(1-\varepsilon_i^{l-1})^2\Big]^{1/2}\,d\varepsilon^1\dotsb d\varepsilon^{l-1}.
\end{align*}
Now applying the Cauchy--Schwarz inequality, we can further estimate it by
\begin{align*}
&\le  \sum_{l=1}^k\Big(\sum_{i\in A} |a_i|^2 \int (1-\varepsilon_i^1)^2\dotsb(1-\varepsilon_i^{l-1})^2 \,d\varepsilon^1\dotsb d\varepsilon^{l-1}\Big)^{1/2}
\\
&= \sum_{l=1}^k\Big(\sum_{i\in A} |a_i|^2 2^{2(l-1)}\mathbb{P}
(\{\varepsilon_i^1=\dots=\varepsilon_i^{l-1}=-1\})\Big)^{1/2}
\\
&= \sum_{l=1}^k\Big(\sum_{i\in A} |a_i|^2 2^{2(l-1)}2^{-(l-1)}\Big)^{1/2}
= \sum_{l=1}^k2^{l/2}\lesssim t,
\end{align*}
which shows \eqref{eq:Phi}.

To prove \eqref{eq:Aeps} we notice that 
$$
|A_\varepsilon|=2^{-k}\sum_{i\in A} (1-\varepsilon_i^1)\dotsb(1-\varepsilon_i^k),
$$
hence due to the independence of $(\varepsilon_i^j)$ we get
$$
\int |A_\varepsilon|\,d\varepsilon=2^{-k}\sum_{i\in A} \Big(\prod_{j=1}^k \int  (1-\varepsilon_i^j)\, d\varepsilon_i^j\Big)=2^{-k}\sum_{i\in A}1=2^{-k}m \simeq \frac{m}{t^2}.
$$

Finally, to show \eqref{eq:E} we use orthogonality and 1-boundedness of $(\varphi_i)_{i=1}^n$ and then the Cauchy--Schwarz inequality getting
\begin{equation*}
\int\|E(\mathbf{\varepsilon},u)\|_{L^2( du)}\,d\varepsilon\le\int \Big[\sum_{i\in A} |a_i|^2 (1-\varepsilon_i^1)^2\dotsb(1-\varepsilon_i^k)^2\Big]^{1/2}\,d\varepsilon\le 2^{k/2} \simeq t.
\end{equation*}
That concludes the proof of Theorem \ref{thm:entropymain2}.
\end{proof}

\section{End of the proof of Theorem \ref{thm:main2}}\label{sec:end}
It this section we complete the proof of Theorem \ref{thm:main2}. We need to treat separately the cases $2<p<4$ and $p\ge 4$.
\subsection{Case $2<p<4$}
In the previous section we reduced the proof of Theorem \ref{thm:main2} to showing the following result. 
\begin{thm}\label{thm:finalp<4}
There exists $\sigma>0$ such that for any $\omega_2, \omega_3$
\begin{align}\label{eqn:3.29-Bourgain}
	\|K_{m_1,m_2,m_3}(\omega_1,\omega_2,\omega_3)\|_{L^{q_0}(d\omega_1)}\lesssim (\delta m_3^{p/2-1}+\frac{m_2+m_3}{m_1})^{1/2}(1+K(\omega_2)+K(\omega_3))^{p-\sigma},
\end{align}
where $q_0=\log n \simeq \log \frac{1}{\delta}, n^{2/p}=\delta n$, with $K_{m_1,m_2,m_3}(\omega_1,\omega_2,\omega_3)$ defined by \eqref{eq:K_{m_1,m_2,m_3}} and $K(\omega)=K_{S_\omega}$ given by \eqref{eq:KS}.
\end{thm}
\begin{proof}

Fix $\omega_2$ and $\omega_3$.
Letting $\mathcal{E}=\big\{\big(|\langle \varphi_i, 	f_{\bm b,\omega_2}(1+|f_{\bm c,\omega_3}|)^{p-2}\rangle	|\big)_{i=1}^{n}:\bm b\in \Pi_{m_2}, \bm c\in \Pi_{m_3}\big\}$ we can write
$$
K_{m_1,m_2,m_3}(\omega_1,\omega_2,\omega_3)= \frac{1}{\sqrt{m_1}}\sup_{|A|\le m_1}\sup_{\mathbf{x}\in \mathcal{E}}\Big(\sum_{i\in A}\xi_i(\omega_1)x_i\Big).
$$
Thus, by Lemma \ref{lem:1lem1} (applied with $m=m_1$) we get 
\begin{align}\label{eqn:4.1}
	\|K_{m_1,m_2,m_3}(\omega_1,\omega_2,\omega_3)\|_{L^{q_0}(d\omega_1)}&\lesssim \left[\delta^{1/2}+m_{1}^{-1/2}\right]B+m_{1}^{-\frac{1}{2}}(\log n)^{-\frac12}\int_{0}^B[\log N_2(\mathcal{E},t)]^{\frac{1}{2}}\,dt,
\end{align}
where $B = \sup_{\mathbf{x}\in \mathcal E}|\mathbf{x}|$. It remains to estimate $B$ and $\log N_2(\mathcal{E},t)$ suitably. Following Bourgain, we will use the shorthand notation
\[
g_{\bm b} = f_{\bm b,\omega_2},\quad h_{\bm c} = f_{\bm c,\omega_3}.
\]

We begin with treating $B$. Take any $\mathbf{x}=\big(|\langle \varphi_i, 	g_{\bm b}(1+|h_{\bm c}|)^{p-2}\rangle	|\big)_{i=1}^n\in\mathcal{E}$. By Bessel's inequality and H\"older's inequality, we have
\begin{align*}
	|\mathbf{x}|=\Big(\sum_{i}|\langle\varphi_i, g_{\bm b}(1+|h_{\bm c}|)^{p-2}\rangle|^2\Big)^{1/2}&\le \|g_{\bm b}(1+|h_{\bm c}|)^{p-2}\|_{2}\notag\\
	&\le \|g_{\bm b}\|_{p}\left\|1+|h_{\bm c}|\right\|_{2p}^{p-2}\notag\\
	&\le \|g_{\bm b}\|_{p} \left(1+\left\|h_{\bm c}\right\|_{\infty}\right)^{\frac{p}{2}-1} \left(1+\left\|h_{\bm c}\right\|_{p}\right)^{\frac{p}{2}-1}\notag\\
	&\lesssim K(\omega_2)(1+K(\omega_3))^{\frac{p}{2}-1}m_3^{\frac{1}{2}(\frac{p}{2}-1)},
\end{align*}
where in the last inequality we used the estimates
$$
\|g_{\bm b}\|_{p}\le K(\omega_2), \qquad \|h_{\bm c}\|_{p}\le K(\omega_3),
$$
and the Cauchy--Schwarz inequality and $1$-boundedness of the system $(\varphi_i)_{i=1}^n$ giving
$$
\|h_{\bm c}\|_{\infty}\le |\bm c|\cdot(\sum_{i\in\mathrm{supp}\,\bm {c}}\|\varphi_i\|_\infty^2)^{1/2}\le m_3^{1/2}.
$$
It follows that
\begin{equation}\label{eqn:4.2}
B\lesssim K(\omega_2)(1+K(\omega_3))^{\frac{p}{2}-1}m_3^{\frac{1}{2}(\frac{p}{2}-1)}.
\end{equation}

We turn to estimating $\log N_2(\mathcal{E},t)$. Let $\mathbf{x}=\big(|\langle \varphi_i, g_{\bm b}(1+|h_{\bm c}|)^{p-2}\rangle|\big)_{i=1}^n$ and $\mathbf{x}'=\big(|\langle \varphi_i, 	g_{\bm b'}(1+|h_{\bm c'}|)^{p-2}\rangle|\big)_{i=1}^n$ be two elements in $\mathcal{E}$. Then by Bessel's inequality, 
\begin{align}\label{eqn:dist}
|\mathbf{x}-\mathbf{x}'|&=\Big(\sum_{i}\left| |\langle\varphi_i, g_{\bm b}(1+|h_{\bm c}|)^{p-2}\rangle|-|\langle\varphi_i,g_{\bm b'}(1+|h_{\bm c'}|)^{p-2}\rangle|\right|^2\Big)^{1/2}\notag
\\
&\le \|g_{\bm b}(1+|h_{\bm c}|)^{p-2}-g_{\bm b'}(1+|h_{\bm c'}|)^{p-2}\|_{2}.
\end{align}
Depending on the value of $p$, we need to estimate the right-hand side of \eqref{eqn:dist} differently. 

Assume first that $2 < p \le 3$. Using the elementary inequality
\begin{align}\label{eqn:psmallerthan3}
|(1+|x|)^{p-2}-(1+|y|)^{p-2}| = \Big|\int_{1+|y|}^{1+|x|} (p-2)s^{p-3}\,ds\Big|\le (p-2)|x-y|, \qquad p\in(2,3),
\end{align}
we get
\begin{align*}
\textrm{RHS\ of \eqref{eqn:dist}} 
&= \|(g_{\bm b}-g_{\bm b'})(1+|h_{\bm c}|)^{p-2}+g_{\bm b'}\big[(1+|h_{\bm c}|)^{p-2}-(1+|h_{\bm c'}|)^{p-2}\big]\|_{2}
\\
&\lesssim \|(g_{{\bm b}}-g_{{\bm b'}})(1+|h_{\bm c}|)^{p-2}\|_{2}+\|g_{{\bm b'}}|h_{\bm c}-h_{\bm c'}|\|_{2}.
\end{align*}
Using H\"older's inequality with $q=\frac{2p}{4-p}$ and $r=\frac{2p}{p-2}$, we further estimate
\begin{align*}
|\mathbf{x}-\mathbf{x}'|
&\le \|g_{{\bm b}}-g_{{\bm b'}}\|_{q}(1+\|h_{{\bm c}}\|_{p})^{p-2}+\|g_{{{\bm b'}}}\|_{p}\|h_{{\bm c}}-h_{{\bm c'}}\|_{r}
\\
&\le (1+K(\omega_3))^{p-2}\|g_{{\bm b}}-g_{{\bm b'}}\|_{q}+K(\omega_2)\|h_{{\bm c}}-h_{{\bm c'}}\|_{r}.
\end{align*}
Consider the function 
\[
T\colon\mathcal P_{m_2}\times\mathcal P_{m_3}\to \R^n
\]
defined by
\[
T(f,g)=\big(|\langle \varphi_i, f(1+|g|)^{p-2}\rangle|\big)_{i=1}^n.
\]
Then 
\[
\mathbf{x}=T(g_{\bm b},h_{\bm c})\qquad\textrm{and}\qquad \mathbf{x}'=T(g_{\bm b'},h_{\bm c'}).
\]
By Proposition \ref{prop:ent-product} applied with $\mathcal E = T(\mathcal P_{m_2}\times\mathcal P_{m_3})$, $c = (1+K(\omega_3))^{p-2}$ and $c' = K(\omega_2)$, it follows that
\begin{align*}
N_2(\mathcal{E},t)\lesssim N_q(\mathcal{P}_{m_2},\frac{t}{2}(1+K(\omega_3))^{2-p})N_r(\mathcal{P}_{m_3},\frac{t}{2}K(\omega_2)^{-1}).
\end{align*}

Hence,
\begin{align}\label{eqn:4.4}
\log N_2(\mathcal{E},t)\lesssim \log N_q(\mathcal{P}_{m_2},\frac{t}{2}(1+K(\omega_3))^{2-p})+ \log N_r(\mathcal{P}_{m_3},\frac{t}{2}K(\omega_2)^{-1}).
\end{align}
Inserting \eqref{eqn:4.2} and \eqref{eqn:4.4} into \eqref{eqn:4.1}, we obtain
\begin{align*}
\|K_{m_1,m_2,m_3}(\omega_1,\omega_2,\omega_3)\|_{L^{q_0}(d\omega_1)}&\lesssim \left[\delta^{1/2}+m_{1}^{-1/2}\right]B+m_{1}^{-\frac{1}{2}}(\log n)^{-1/2}\int_{0}^B[\log N_2(\mathcal{E},t)]^{\frac{1}{2}}\,dt
\\
&\lesssim \left[\delta^{1/2}+m_{1}^{-1/2}\right]m_3^{\frac{1}{2}(\frac{p}{2}-1)}K(\omega_2)(1+K(\omega_3))^{\frac{p}{2}-1}
\\
&\quad+m_{1}^{-1/2}(\log n)^{-1/2}K(\omega_2)\int_{0}^{\infty} \log [N_q(\mathcal{P}_{m_2},t)]^{\frac{1}{2}}\,dt
\\
&\quad+m_{1}^{-1/2}(\log n)^{-1/2}(1+K(\omega_3))^{p-2}\int_{0}^{\infty} \log [N_r(\mathcal{P}_{m_3},t)]^{\frac{1}{2}}\,dt. 
\end{align*}
Next, we can apply Theorem \ref{thm:entropymain} getting
\begin{align*}
&\|K_{m_1,m_2,m_3}(\omega_1,\omega_2,\omega_3)\|_{L^{q_0}(d\omega_1)}
\\
&\quad\lesssim
\left[\delta m_3^{\frac{p}{2}-1}+\frac{m_3}{m_1}\right]^{\frac{1}{2}}K(\omega_2)(1+K(\omega_3))^{\frac{p}{2}-1}+ \left(\frac{m_2}{m_1}\right)^{\frac{1}{2}}(1+K(\omega_3))^{{p}-2}+\left(\frac{m_3}{m_1}\right)^{\frac{1}{2}}K(\omega_2),
\end{align*}
which implies \eqref{eqn:3.29-Bourgain} with $\sigma=\frac{p}{2}$, since $K(\omega_3)\ge 1$ by \eqref{eqn:K-at-least-1}. It finishes the proof in the case $2 < p \le 3$.

For $3<p<4$, the inequality \eqref{eqn:psmallerthan3} is no longer valid. Instead, we use 
\begin{equation}\label{pgreaterthan3}
    |(1+|x|)^{p-2}-(1+|y|)^{p-2}|
    \le (p-2)|x-y|(|x|^{p-3}+|y|^{p-3}), \qquad p>3.
\end{equation}
Applying the above inequality we get
\begin{align*}
    |g_{\bm b}(1+|h_{\bm c}|)^{p-2}-g_{{\bm b}'}(1+|h_{{\bm c}'}|)^{p-2}| &= |(g_{\bm b}-g_{{\bm b}'}+g_{{\bm b}'})(1+|h_{\bm c}|)^{p-2}-g_{{\bm b}'}(1+|h_{{\bm c}'}|)^{p-2}| \\
    &\lesssim |g_{\bm b}-g_{{\bm b}'}|(1+|h_{\bm c}|)^{p-2}
\\
&\quad+|g_{{\bm b}'}| |h_{\bm c}-h_{{\bm c}'}| (|h_{\bm c}|^{p-3}+|h_{{\bm c}'}|^{p-3}).
\end{align*}
Hence, with $q=2p/(4-p)$,  we can compute the distance by
\begin{align*}
    \textrm{RHS\ of \eqref{eqn:dist}} &\lesssim \|g_{\bm b}-g_{{\bm b}'}\|_{q}(1+\|h_{\bm c}\|_p)^{p-2}+\|g_{{\bm b}'}\|_p \|h_{\bm c}-h_{{\bm c}'}\|_q (\|h_{\bm c}\|_p^{p-3}+\|h_{{\bm c}'}\|_p^{p-3})\\
    &\lesssim K(\omega_3)^{p-2} \|g_{\bm b}-g_{{\bm b}'}\|_{q}+ K(\omega_2)K(\omega_3)^{p-3}\|h_{\bm c}-h_{{\bm c}'}\|_q
\end{align*}
using the general form of H\"older's inequality with exponents $(\frac{p}{2}, \frac{p}{4-p}, \frac{p}{2(p-3)})$ for the second term.
We can follow the rest of the argument from the  case $2<p\le 3$ to finish the proof.
\end{proof}

\subsection{Case $p\ge 4$}
To finish the proof of Theorem \ref{thm:main2} it suffices to treat the case $p\ge 4$. The arguments presented in Section \ref{subsec:p>4} reduced the problem to showing the following result. 
\begin{thm}
Let $p\ge4$ and let $p_1$ satisfy $p/2< p_1 < p$. Assume that \eqref{eqn:lambda-p1} holds for some set $S_1\subset[n]$ of cardinality $|S_1|\simeq n^{2/p_1}$ . Then
there exists $\sigma>0$ such that for any $\omega_2, \omega_3$
the estimate \eqref{eq:Kest-p>4} holds with $\delta'=n^{2/p-2/p_1}$.
\end{thm}
\begin{proof}
The argument is similar to the proof of Theorem \ref{thm:finalp<4}. Fix $\omega_2$ and $\omega_3$. Letting
$$
\mathcal{E}_{S_1}:=\big\{\big(|\langle \varphi_i, f_{\bm b,\omega_2}(1+|f_{\bm c,\omega_3}|)^{p-2}\rangle	|\big)_{i\in S_1}:\bm b\in \Pi^{S_1}_{m_2}, \bm c\in \Pi^{S_1}_{m_3}\big\}
$$ 
we see that
$$
K^{S_1}_{m_1,m_2,m_3}(\omega_1,\omega_2,\omega_3)=\frac{1}{\sqrt{m_1}}\sup_{\substack{A\subset S_1\\|A|\le m_1}}\sup_{\mathbf{x}\in \mathcal{E}_{S_1}}\Big(\sum_{i\in A}\xi_i(\omega_1)x_i\Big).
$$
Thus, by Lemma \ref{lem:1lem1} (applied with $m=m_1$) we get
\begin{align*}
	\|K^{S_1}_{m_1,m_2,m_3}(\omega_1,\omega_2,\omega_3)\|_{L^{q_0}(d\omega_1)}&\lesssim \left[{\delta'}^{1/2}+m_{1}^{-1/2}\right]B
\\
&\quad+m_{1}^{-\frac{1}{2}}(\log n)^{-1/2}\int_{0}^B[\log N_2(\mathcal{E}_{S_1},t)]^{\frac{1}{2}}\,dt,
\end{align*}
where $B = \sup_{\mathbf{x}\in \mathcal E_{S_1}}|\mathbf{x}|$. It remains to estimate  $B$ and $\log N_2(\mathcal{E}_{S_1},t)$.
Let 
\[
g_{\bm b} = f_{\bm b,\omega_2},\quad h_{\bm c} = f_{\bm c,\omega_3}.
\]

We begin with treating $B$. Note that the induction hypothesis \eqref{eqn:lambda-p1} is equivalent to the following dual form of it
\begin{equation}\label{eqn:dual-lambda-p1}
    \Big(\sum_{i\in S_1}|\langle\varphi_i, \psi\rangle|^2\Big)^{1/2} \lesssim \|\psi\|_{p_1'}, \qquad \psi\in L^{p_1'}(du).
\end{equation}
Now we will argue as in the proof of Theorem \ref{thm:finalp<4}, but instead of using Bessel's inequality we will use \eqref{eqn:dual-lambda-p1}.
Letting $\mathbf{x}=\big(|\langle \varphi_i, 	g_{\bm b}(1+|h_{\bm c}|)^{p-2}\rangle|\big)_{i\in S_1}\in\mathcal{E}_{S_1}$
we get
\begin{align*}
   |\mathbf{x}|&= \Big(\sum_{i\in S_1}|\langle\varphi_i, g_{\bm b}(1+|h_{\bm c}|)^{p-2}\rangle|^2\Big)^{1/2} \lesssim \|g_{\bm b}(1+|h_{\bm c}|)^{p-2}\|_{p_1'}\\
    &\lesssim \|g_{\bm b}\|_p (1+\|h_{\bm c}\|_{p})^{p/{p_1'}-1} (1+\|h_{\bm c}\|_\infty)^{p/{p_1}-1}\lesssim K(\omega_2)(1+K(\omega_3))^{p/{p_1'}-1}m_3^{\frac{1}{2}(\frac{p}{p_1}-1)},
\end{align*}
where \eqref{eqn:dual-lambda-p1} was used in the first inequality, and then H\"older's inequality was used with exponents $(\frac{p}{p_1'}, \frac{p}{p-p_1'})$. That completes the estimate of $B$.

Next we treat $\log N_2(\mathcal{E}_{S_1},t)$. For $\mathbf{x}, \mathbf{x}'\in \mathcal{E}_{S_1}$ with $\mathbf{x}=\big(|\langle \varphi_i, g_{\bm b}(1+|h_{\bm c}|)^{p-2}\rangle|\big)_{i\in S_1}$ and $\mathbf{x}'=\big(|\langle \varphi_i, 	g_{\bm b'}(1+|h_{\bm c'}|)^{p-2}\rangle|\big)_{i\in S_1}$ we get by \eqref{eqn:dual-lambda-p1} 
\begin{align*}
|\mathbf{x}-\mathbf{x}'| &= \Big(\sum_{i\in S_1}\left| |\langle\varphi_i, g_{\bm b}(1+|h_{\bm c}|)^{p-2}\rangle|-|\langle\varphi_i,g_{\bm b'}(1+|h_{\bm c'}|)^{p-2}\rangle|\right|^2\Big)^{1/2} \\
&\lesssim\|g_{\bm b}(1+|h_{\bm c}|)^{p-2}-g_{\bm b'}(1+|h_{\bm c'}|)^{p-2}\|_{p_1'}.
\end{align*}
Using the triangle inequality in $L^{p_1'}$ and inequality \eqref{pgreaterthan3} we get
\begin{equation*}
|\mathbf{x}-\mathbf{x}'|\lesssim \|(g_{\bm b}-g_{\bm b'})(|h_{\bm c}|^{p-2}+|h_{{\bm c}'}|^{p-2})\|_{p_1'}+\|g_{\bm b'}|h_{\bm c}-h_{{\bm c}'}| (|h_{\bm c}|^{p-3}+|h_{\bm c'}|^{p-3}) \|_{p_1'}.
\end{equation*}
Applying H\"older's inequality with exponents $(\frac{q}{p_1'},\frac{p}{p_1'(p-2)})$, where $q=pp_1'/(p-p_1'(p-2))$, for the first term on the right--hand side, and with exponents $(\frac{p}{p_1'},\frac{q}{p_1'},\frac{p}{p_1'(p-3)})$ for the second term, we get
\begin{align*}
|\mathbf{x}-\mathbf{x}'| &\lesssim \|g_{\bm b}-g_{\bm b'}\|_q (\|h_{\bm c}\|_p +  \|h_{{\bm c}'}\|_p)^{p-2} + \|g_{\bm b'}\|_p\|h_{\bm c}-h_{{\bm c}'}\|_q(\|h_{\bm c}\|_p + \|h_{\bm c'}\|_p)^{p-3}.
\end{align*}
Finally, using the definition of $K(\omega_2),K(\omega_3)$, and the elementary inequality $ab^{p-3}\le a^{p-2} + b^{p-2} \le C_p(a+b)^{p-2}$, valid for $a,b\ge 0$, we obtain
\begin{align*}
|\mathbf{x}-\mathbf{x}'|&\lesssim K(\omega_3)^{p-2}\|g_{\bm b}-g_{\bm b'}\|_q+ K(\omega_2)K(\omega_3)^{p-3}\|h_{\bm c}-h_{{\bm c}'}\|_q 
\\
&\lesssim (K(\omega_2)+K(\omega_3))^{p-2}(\|g_{\bm b}-g_{{\bm b}'}\|_q+\|h_{\bm c}-h_{{\bm c}'}\|_q).
\end{align*}
Therefore invoking Proposition \ref{prop:ent-product} to bound $N_2(\mathcal E_{S_1},t)$, we have the estimate
\[
\int_0^B [\log N_2(\mathcal E_{S_1},t)]^{1/2}\,dt \lesssim (K(\omega_2)+K(\omega_3))^{p-2}(m_2+m_3)^{1/2}(\log n)^{1/2}.
\]

Combining the above estimates for $B$ and $\int_0^B [\log N_2(\mathcal E_{S_1},t)]^{1/2}\,dt$, by Lemma \ref{lem:1lem1}, and the relations $\log\frac{1}{\delta'}\simeq\log n$ and $\frac{p}{p_1}-1<1$, we get
\begin{align*}
    \| K^{S_1}_{m_1,m_2,m_3}(\omega_1,\omega_2,\omega_3)\|_{L^{q_0}(d\omega_1)} &\lesssim \left[\delta'+m_{1}^{-1}\right]^{1/2}B+m_{1}^{-\frac{1}{2}}(\log n)^{-\frac12}\int_{0}^B[\log N_2(\mathcal{E}_{S_1},t)]^{\frac{1}{2}}\,dt\\
    &\lesssim \big(\delta'm_3^{\frac{p}{p_1}-1}+\frac{m_3}{m_{1}}\big)^{1/2}K(\omega_2)(1+K(\omega_3))^{\frac{p}{p_1'}-1}\\
    &\quad+(K(\omega_2)+K(\omega_3))^{p-2}(\frac{m_2+m_3}{m_1})^{1/2}.
\end{align*}
Finally, we observe $K(\omega_2)(1+K(\omega_3))^{p/p_1'-1}\lesssim (K(\omega_2) + K(\omega_3))^{p/p_1'} \le (K(\omega_2) + K(\omega_3))^{p-1}$, where we apply \eqref{eqn:K-at-least-1} in the first inequality. This shows that \eqref{eq:Kest-p>4} holds with $\sigma = 1$.
\end{proof}
\appendix
\section{Technical results}
\begin{lem}\label{lem:F}
Let $q>1$ and let $0< \kappa <q$. Consider the function
$$
F(x)=\Big(\frac{\kappa}{x}\Big)^x x^q.
$$
Then there exists $x_0\in (1,\infty)$ such that $F$ is increasing on $(1,x_0)$  and decreasing on $(x_0,\infty)$.
\end{lem}

\begin{proof}
We start with writing
$$
F(x)=e^{x\log \frac{\kappa}{x}} x^q.
$$
Then the derivative of $F$ is
$$
F'(x)=x^{q-1} e^{x\log \frac{\kappa}{x}} \Big[x\big(\log \frac{\kappa}x -1\big)+q \Big]=x^{q-1} e^{x\log \frac{\kappa}{x}} \Big[x\log \frac{\kappa}{x e}+q \Big].
$$
Thus $F'(x)=0$ for some $x\in (1,\infty)$ if and only if 
$$
x\log\frac{x e}{\kappa}=q, \qquad x\in (1,\infty).
$$
Letting $K=\frac{e}{\kappa}$ and $u=Kx$ the above equation is equivalent to 
$$
u\log u=q K, \qquad u\in (K,\infty).
$$
It remains to note that the function $u\mapsto u\log u$ has negative values for $u$ in $(0, \frac{1}e)$ and it increases unboundedly on $(\frac{1}e, \infty)$. Therefore the equation $u\log u=q K$ has exactly one solution and clearly the desired conclusion about $F$ follows.
\end{proof}

\subsection{Large deviations}
Suppose $\{\xi_i:i\in [n]\}$ are independent $\{0,1\}$-valued random variables of mean $\delta(n)$, which is allowed to depend on $n$. We can think of $\delta = n^{2/p-1}$ as it is in Bourgain's paper---something which for fixed $p$ decays much slower than $n^{-1}$ as $n\to\infty$. Set 
$$
S_n = \sum_{i=1}^n\xi_i
$$
and note that $\E S_n=n\delta$.
A typical large deviations result for the family $\{\xi_i:i\in[n]\}$ would establish an exponential rate of decay of the probability of the rare event that $S_n/n$ is larger than 10 times its expected value: 
\[
\mathbb{P}(S_n/n > 10\delta) \le e^{-c n}.
\]
Many families of independent random variables enjoy such a property as long as $\delta$ stays sufficiently large compared to $n^{-1}$. 

\begin{prop}\label{prop:largedev}
For all $n\ge 1$, the following estimates hold.
\begin{itemize}
\item[(i)] $\mathbb{P}(S_n>10n\delta) \le e^{-5n\delta}$
\item[(ii)] $\mathbb P(S_n<\frac{1}{10}n\delta) \le e^{-\frac{n\delta}{2}}$.
\end{itemize}
\end{prop}
\begin{proof}
First we prove (i). Let $t > 0$ be an extra parameter we have at our disposal. By Chebyshev's inequality and independence of the selectors,
\begin{align*}
p_n := \mathbb{P}[\sum_{i=1}^n\xi_i > 10n\delta] &\le \mathbb{P}[e^{t\sum_{i=1}^n\xi_i}\ge e^{10tn\delta }] \\
&\le e^{-10tn\delta }\E[e^{t\sum_{i=1}^n\xi_i}] \\
&= e^{-10t n\delta }\prod_{i=1}^n \E e^{t\xi_i}.
\end{align*}
A direct computation shows
\[
\E e^{t\xi_i} = e^t\delta + (1-\delta) = 1 + (e^t-1)\delta.
\]
By what we have so far,
\[
p_n \le e^{-10tn\delta } [1+(e^t-1)\delta]^n \le e^{-10 t n\delta } e^{(e^t-1)n\delta} = [e^{-10t+(e^t-1)}]^{n\delta}.
\]
We could optimize to choose the best value of $t$, but setting $t = 1$ is sufficient because it shows
\[
p_n \le [e^{-8.28...}]^{n\delta}\le e^{-5n\delta}.
\]

The proof of (ii) is very similar. First we transform the expression so it more closely resembles what we did to prove (i). Let $\eta_i = \delta-\xi_i$, and let $T_n = \sum_{i=1}^n \eta_i$. Then the reader can easily verify
\[
\mathbb P(S_n<\frac{1}{10}n\delta) = \mathbb P(T_n > \frac{9}{10}n\delta).
\]
We introduce a free parameter $t$ as before, and use Chebyshev and independence of the variables $\eta_i$:
\begin{align*}
\mathbb P(T_n>\frac{9}{10}n\delta) &\le e^{-\frac{9}{10}tn\delta}\prod_{i=1}^n\mathbb E e^{t\eta_i}\\
&= e^{-\frac{9}{10}tn\delta}(e^{t\delta}(1-\delta)+e^{t(\delta-1)}\delta)^n\\
&= e^{\frac{1}{10}tn\delta}(1+(e^{-t}-1)\delta)^n\\
&\le e^{\frac{1}{10}tn\delta}e^{(e^{-t}-1)n\delta}\\
&= [e^{\frac{t}{10}+(e^{-t}-1)}]^{n\delta}.
\end{align*}
Setting $t = 1$, we have
\[
\mathbb P(T_n>\frac{9}{10}n\delta) \le [e^{-0.532...}]^{n\delta} \le e^{-\frac{n\delta}{2}}.
\]
\end{proof}

\subsection{Entropy of products} Suppose $(X,d)$, $(X',d')$ are two metric spaces, $c,c'>0$, and consider the product metric space $(X\times X',cd\oplus c'd')$ with metric defined by
\[
(cd\oplus c'd')((x_1,x_1'),(x_2,x_2')) = cd(x_1,x_2) + c'd'(x_1',x_2').
\]
Lipschitz images of such product metric spaces obey natural entropy bounds:
\begin{prop}\label{prop:ent-product}
Let $(X,d)$, $(X',d')$ and $(\bar X,\bar d)$ be metric spaces. Assume there exists a $1$-Lipschitz map $T\colon (X\times X',cd\oplus c'd')\to (\bar X,\bar d)$, meaning that for any $(x_1,x_1'), (x_2,x_2') \in X\times X'$,
    \[
    \bar d(T(x_1,x_1'), T(x_2,x_2')) \le (cd\oplus c'd')((x_1,x_1'), (x_2,x_2')).
    \] 
    Then the covering number $N_{\bar d}(T(X\times X'),t)$ satisfies the bound
    \[
    N_{\bar d}(T(X\times X'),t) \lesssim  N_d(X,\frac{t}{2c})\cdot  N_{d'}(X',\frac{t}{2c'})
    \]
    for any $t > 0$.
\end{prop}
\begin{proof}
    Let $\mathcal B = \{B_i\}$ and $\mathcal B' = \{B_j'\}$ be minimal coverings of $X,X'$ by $\frac{t}{2c},\frac{t}{2c'}$-balls in their respective metrics. For fixed $i,j$, we claim that $T(B_i\times B_j')$ is contained in a $t$-ball of $\bar X$. To see this,  if $T(x,x')\in T(B_i\times B_j')$, then with $B_i = B(x_i,\frac{t}{2c})$ and $B_j' = B(x_j',\frac{t}{2c'})$, by assumption,
    \[
    \bar d(T(x,x'), T(x_i,x_j')) \le cd(x,x_i) + c'd'(x',x_j') < c\cdot \frac{t}{2c} + c'\cdot \frac{t}{2c'} = t.
    \]
    Since $T(X\times X')\subset \bigcup_{i,j}T(B_i\times B_j')$, it follows from the definition of the entropy numbers that the minimum number of $t$-balls required to cover $T(X\times X')$ is
    \[
     N_{\bar d}(T(X\times X'),t) \lesssim|\mathcal B|\cdot|\mathcal B'| \simeq  N_d(X,\frac{t}{2c})\cdot  N_{d'}(X',\frac{t}{2c'}),
    \]
    as desired.
\end{proof}

\subsection{Level set decomposition of a sequence}
The result below is a version of a level set partition. Its proof is based on the Cauchy condensation test.

\begin{prop}\label{prop:seqdyad}

Let $\bm{c}=(c_1, c_2, \dots)$ be a nonincreasing sequence of nonnegative numbers satisfying 
$$
\sum_{i=1}^\infty c_i \le 1.
$$ 
Then there exists a family of sequences $\bm{c}(l)=(c_1(l), c_2(l), \dots)$, $l=0,1, \dots$ with mutually disjoint supports and a sequence of nonnegative coefficients $(\gamma_l)_{l=0}^\infty$ such that 
\begin{align*}
\bm{c} = \sum_{l=0}^\infty \gamma_{l}\bm{c}(l)
\end{align*}
and the following conditions hold
\begin{itemize}
\item[1)] $\sum_{l=0}^\infty \gamma_l\le 1,$ \label{cond1}
\item[2)] $|\supp\bm{c}(l)|\le 2^{l},$\label{cond2}
\item[3)] $c_i(l)\le 2^{-l}.$\label{cond3}
\end{itemize}
\end{prop}

\begin{proof}
Note that
\begin{align*}
\sum_{i=1}^\infty c_i = \sum_{l=0}^\infty \sum_{i=2^l}^{2^{l+1}-1} c_i 
= \sum_{l=0}^\infty 2^l c_{2^l}\sum_{i=2^l}^{2^{l+1}-1} \frac{c_i}{2^l c_{2^l}} 
=: \sum_{l=0}^\infty \gamma_l \sum_{i=2^l}^{2^{l+1}-1} c_i(l), 
\end{align*}
where $\gamma_l = 2^l c_{2^l}$ and
$$
c_i(l)=
\begin{cases}
\displaystyle\frac{c_i}{2^lc_{2^l}}, &\qquad i\in [2^l, 2^{l+1}-1],\\
0, &\qquad i\notin [2^l, 2^{l+1}-1].
\end{cases}
$$ 
It remains to notice that using monotonicity of $c_i$ we get
$$
\sum_{l=0}^\infty \gamma_l =\sum_{l=0}^\infty 2^l c_{2^l}\le \sum_{l=0}^\infty \sum_{i=2^l}^{2^{l+1}-1} c_i\le 1,
$$
so condition 1) holds. Moreover, condition 2) holds trivially. Finally, condition 3) holds since $c_i\le c_{2^l} \le 1$ for $i\in \supp\bm{c}(l)=[2^l, 2^{l+1}-1]$.
\end{proof}

\begin{cor}\label{cor:seqdyad}
Let $\bm{a}=(a_1, a_2, \dots)$ be a sequence of complex numbers satisfying $|\bm{a}|\le 1$.
Then there exists a family of sequences $\bm{a}(l)=(a_1(l), a_2(l), \dots)$, $l=0,1, \dots$ with mutually disjoint supports and a sequence of nonnegative coefficients $(\lambda_l)_{l=0}^\infty$ such that 
\begin{align*}
\bm{a} = \sum_{l=0}^\infty \lambda_{l}\bm{a}(l)
\end{align*}
and the following conditions hold
\begin{itemize}
\item[A)] $\sum_{l=0}^\infty \lambda_l^2 \le 1,$
\item[B)] $|\supp\bm{a}(l)|\le 2^{l},$
\item[C)] $|a_i(l)|\le 2^{-l/2}.$
\end{itemize}

\begin{proof}
Consider the sequence $c_i: = |a_i|^2$ of nonnegative numbers. By re-indexing we can assume that $c_1\ge c_2\ge \dotsc$. Using Proposition \ref{prop:seqdyad} there exists a decomposition 
$$
\bm{c} = \sum_{l=0}^\infty \gamma_{l}\bm{c}(l)
$$
with $\gamma_l$ and $\bm{c}(l)$ satisfying conditions $1)-3)$ of Proposition \ref{prop:seqdyad}. Then letting 
$$
a _i(l)=
\begin{cases}
\frac{a_i}{2^{l/2} |a_{2^l}|}, &\qquad i\in [2^l, 2^{l+1}-1],\\
0, &\qquad i\notin [2^l, 2^{l+1}-1],
\end{cases}
$$
and $\lambda_l= \gamma_l^{1/2}$ it is straightforward to see that conditions A)---C) hold.
\end{proof}
\end{cor}

\end{document}